\documentclass[10pt,oneside]{article}

\usepackage[utf8]{inputenc} %
\usepackage[T1]{fontenc}    %
\usepackage{url}            %
\usepackage{booktabs}       %
\usepackage{amsfonts}       %
\usepackage{nicefrac}       %
\usepackage{microtype}      %

\usepackage{lmodern}
\usepackage{times}

\usepackage{amssymb,amsmath,amsthm}
\usepackage{bbold}
\usepackage[short]{optidef}

\usepackage{framed,multirow,multicol}

\usepackage{xcolor,xspace}
\usepackage{lscape}
\usepackage{subfigure,graphicx,epsfig,tikz,caption}
\usepackage[normalem]{ulem}
\usepackage{enumerate}
\usepackage{verbatim}
\usepackage{makeidx,latexsym}
\usepackage[colorlinks=true,citecolor=blue]{hyperref}%

\usepackage{bm}
\usepackage{stmaryrd}
\usepackage{lscape}

\usepackage[numbers,sort&compress,square,comma]{natbib}

\usepackage{algorithmic,algorithm}

\usepackage[english]{babel}
\usepackage{bbm}
\usepackage{pgfplots}

    \usepackage{parskip}
\usepackage[letterpaper,margin=1.1in]{geometry}

\usepackage{thmtools,thm-restate}

    \declaretheorem{lemma}
    \declaretheorem{proposition}

    \declaretheoremstyle[qed=$\square$]{definitionwithend}

    \declaretheorem[style=definitionwithend]{remark}

\DeclareMathOperator{\DR}{DR}
\DeclareMathOperator{\SAA}{SAA}

\DeclareMathOperator{\dist}{dist}
\DeclareMathOperator{\intt}{int}
\DeclareMathOperator{\cl}{cl}

\DeclareMathOperator{\CVaR}{CVaR}

\newcommand{\vc}{\mathbf{c}}

\newcommand{\va}{\mathbf{a}}

\newcommand{\vr}{\mathbf{r}}

\newcommand{\vz}{\mathbf{z}} %
\newcommand{\vx}{\mathbf{x}} %
\newcommand{\vy}{\mathbf{y}} %
\newcommand{\vpi}{\bm{\pi}} %

\newcommand{\vmu}{\bm{\mu}} %
\newcommand{\vb}{\mathbf{b}} %

\newcommand{\vxi}{\bm{\xi}}

\newcommand{\bbE}{\mathbb{E}}

\newcommand{\bbN}{\mathbb{N}}

\newcommand{\bbP}{\mathbb{P}}

\newcommand{\bbR}{\mathbb{R}}

\newcommand{\R}{\mathbb{R}}

\newcommand{\cF}{\mathcal{F}}

\newcommand{\cS}{\mathcal{S}}

\newcommand{\cX}{\mathcal{X}}

\newcommand{\epr}{\hfill\hbox{\hskip 4pt \vrule width 5pt height 6pt depth 1.5pt}\vspace{0.0cm}\par}

\def\true{1}

\def\flagJournal{0}

\newcommand {\beqn}{\begin{equation}}\newcommand {\eeqn}{\end{equation}}
\newcommand {\beqan}{\begin{eqnarray}}\newcommand {\eeqan}{\end{eqnarray}}
\newcommand {\beqa}{\begin{eqnarray*}}\newcommand {\eeqa}{\end{eqnarray*}}
\newcommand{\skipit}  [1] {}

\newtheorem{theorem}{Theorem}%

\begin{document}
\title{Distributionally Robust Chance-Constrained Programs  with Right-Hand Side Uncertainty under Wasserstein Ambiguity}
\author{
	Nam Ho-Nguyen\thanks{Wisconsin Institute for Discovery, University of Wisconsin-Madison, Madison, WI 53715, \url{honguyen@wisc.edu}}
	\and
	Fatma K{\i}l{\i}n\c{c}-Karzan\thanks{Tepper School of Business, Carnegie Mellon University, Pittsburgh, PA 15213, USA, \url{fkilinc@andrew.cmu.edu}}
	\and 
	Simge K\"{u}\c{c}\"{u}kyavuz\thanks{Department of Industrial Engineering and Management Sciences, Northwestern University, Evanston, IL 60208, USA, \url{simge@northwestern.edu}}
	\and
	Dabeen Lee\thanks{Discrete Mathematics Group, Institute for Basic Science (IBS), Daejeon 34126, Republic of Korea, \url{dabeenl@ibs.re.kr}}
}	

\date{\today}

\maketitle

\begin{abstract}
We consider exact deterministic mixed-integer programming (MIP) reformulations of distributionally robust chance-constrained programs (DR-CCP) with random right-hand sides over Wasserstein ambiguity sets. The existing MIP formulations are known to have weak continuous relaxation bounds, and, consequently, for hard instances with small radius, or with %
large {problem sizes}, %
the branch-and-bound based solution processes suffer from large optimality gaps even after hours of computation time. This significantly hinders the practical application of the DR-CCP paradigm. Motivated by these challenges, we conduct a polyhedral study to strengthen these formulations. We reveal several hidden connections between DR-CCP and its nominal counterpart (the sample average approximation), mixing sets, and robust 0-1 programming. By exploiting these connections in combination, we provide an improved formulation and two classes of valid inequalities for DR-CCP. We test the impact of our results on a stochastic transportation problem numerically.  Our experiments demonstrate the effectiveness of our approach; in particular our improved formulation and proposed valid inequalities reduce the overall solution times remarkably.  Moreover, this allows us to significantly scale up the {problem sizes} %
that can be handled in such DR-CCP formulations {by reducing the solution times from hours to seconds}. %
\end{abstract}

\section{Introduction}\label{sec:intro}

We consider the following chance-constrained  program (CCP)
\begin{align}
\min_{\vx} \quad & \vc^\top \vx\notag\\ 
\text{s.t.}\quad & \bbP^*[\vxi \not\in \cS(\vx)] \leq \epsilon,\tag{CCP}\label{eq:ccp}\\
& \vx \in \mathcal{X},\notag
\end{align}
where $\vc\in \bbR^L$ is a cost vector, $\cX \subset \bbR^L$ is a compact domain for the decision variables $\vx$, $\cS(\vx) \subseteq \bbR^K$ is a decision-dependent safety set, $\vxi \in \bbR^K$ is a random variable with distribution $\bbP^*$, and $\epsilon \in (0,1)$ is the risk tolerance for the random variable $\vxi$ falling outside the safety set $\cS(\vx)$.

CCP is one of the most common models to handle uncertainty in optimization. Nevertheless, in practice, the distribution $\bbP^*$ in the chance constraint in \eqref{eq:ccp} is often unavailable to the optimizer. Instead, independent and identically distributed (i.i.d.) samples $\{ \vxi_i \}_{i \in [N]}$, where $[N]:=\{1,\ldots,N\}$, are drawn from $\bbP^*$, and $\bbP^*$ is approximated using the empirical distribution $\bbP_N$ on these samples. Such an approach is known as the \emph{sample average approximation} (SAA) of \eqref{eq:ccp}. Note that evaluating $\bbP^*[\vxi \not \in \cS(\vx)]$ exactly is often difficult even when $\bbP^*$ is available. Consequently,  the SAA approach is often employed whenever the computation of $\bbP^*[\vxi \not \in \cS(\vx)]$ is expensive, even if $\bbP^*$ is available. The SAA formulation of \eqref{eq:ccp} is
\begin{align}
\min\limits_{\vx} \quad & \vc^\top \vx\notag\\ 
\text{s.t.}\quad & \frac{1}{N} \sum_{i \in [N]} \bm{1}(\vxi_i \not\in \cS(\vx)) \leq \epsilon,\tag{SAA}\label{eq:saa-ccp}\\
& \vx \in \mathcal{X},\notag
\end{align}
where $\bm{1}(\cdot)$ is the indicator function. 
For certain forms of safety sets $\cS(\cdot)$, \eqref{eq:saa-ccp} can be reformulated as a  mixed-integer program (MIP) and thus off-the-shelf optimization solvers can be used to solve it.

While there are statistical guarantees for using \eqref{eq:saa-ccp} to approximate \eqref{eq:ccp} \cite{Cal1,campi,sample}, the out-of-sample performance of the solution from \eqref{eq:saa-ccp} is quite sensitive to the specific sample $\{ \vxi_i \}_{i \in [N]}$, and can result in high variance, particularly for small $N$. In order to remedy this and regularize the out-of-sample performance of \eqref{eq:saa-ccp}, one can solve a \emph{distributionally robust chance-constrained program}:
\begin{align}
\min\limits_{\vx} \quad & \vc^\top \vx\notag\\ 
\text{s.t.}\quad & \sup_{\bbP \in \cF_N(\theta)} \bbP[\vxi \not\in \cS(\vx)] \leq \epsilon\tag{DR-CCP},\label{eq:dr-ccp}\\
& \vx \in \mathcal{X},\notag
\end{align}
where $\cF_N(\theta)$ is an ambiguity set of distributions on $\bbR^K$ that contains the empirical distribution $\bbP_N$, and $\theta$ is a parameter that governs the size of the ambiguity set, and thereby the conservatism of \eqref{eq:dr-ccp}. For a recent comprehensive survey on distributionally robust optimization problems, their properties, and exact and approximate methods to solve them, we refer the reader to \citet{Rahimian2019DistributionallyRO} and references therein.

Several types of ambiguity sets of probability distributions, such as those based on moments, $\phi$-divergences, unimodality, or support have been studied in the literature; see e.g., \cite{EOO03,JG16,calafiore:jota06,Hanasusanto2015,Li2019,Xie18}. More recently, the \emph{Wasserstein} ambiguity set, that is, the Wasserstein distance ball of radius $\theta$ around the empirical distribution $\bbP_N$, is shown to possess particularly attractive statistical properties. The dual representation for the worst-case probability $\bbP[\vxi \not\in \cS(\vx)]$ under the Wasserstein ambiguity set $\bbP \in \cF_N(\theta)$ is given in \cite{gao2016distributionally,BlanchetMurthy2019,MohajerinEsfahaniKuhn2018}. Various studies~\cite{HotaEtAl2019,xie2018distributionally,chen2018data} exploit this dual representation to give a deterministic non-convex reformulation of \eqref{eq:dr-ccp}. See also \citet{HotaEtAl2019} for reformulations based on the conditional value-at-risk (CVaR) inner approximation of \eqref{eq:dr-ccp} for several different types of safety sets.

For common linear forms of safety sets $\cS(\cdot)$, \citet{chen2018data} and \citet{xie2018distributionally} show that \eqref{eq:dr-ccp} under Wasserstein ambiguity sets admits a MIP reformulation. 
\citet{JiLejeune2019} also explore MIP formulations of \eqref{eq:dr-ccp} under Wasserstein ambiguity. However, they impose additional structure on the support of $\vxi$, thus their formulations are different from \citet{chen2018data} and \citet{xie2018distributionally}. 
Such MIP reformulations pave the path of using standard optimization solvers to formulate and solve these problems, however, according to \cite{chen2018data,xie2018distributionally}  these MIP reformulations are  difficult to solve in certain cases. Even for relatively small size problems with only a few hundred samples and small radii, the resulting formulations still have a large optimality gap even after an hour of computation time with a commercial solver. 

\subsection{Contributions}

Motivated by these computational challenges, our focus in this paper is on developing effective methods to solve the exact reformulation of \eqref{eq:dr-ccp} for such hard cases. 
In particular, we closely examine the MIP reformulations of \eqref{eq:dr-ccp} with random right-hand side uncertainty under Wasserstein ambiguity sets from \citet{chen2018data} and \citet{xie2018distributionally}. We present theoretical results that have strong computational impact in solving \eqref{eq:dr-ccp}.

We first exploit the close relationship between \eqref{eq:dr-ccp} over Wasserstein ambiguity sets and its nominal counterpart \eqref{eq:saa-ccp} to identify an implied mixing set with a cardinality constraint over the existing binary variables in the \eqref{eq:dr-ccp} reformulation, which enables the use of existing inequalities for such a set. We use the established technique of {\emph{quantile strengthening}  \cite{luedtke2010integer,luedtke2014branch-and-cut}} to significantly reduce the big-$M$ constants of this mixing set, and then show how to adapt these to \eqref{eq:dr-ccp}. Our result also shows that existing mixing inequalities for \eqref{eq:saa-ccp} can be readily applied to \eqref{eq:dr-ccp}.

We further analyze the formulation with reduced coefficients obtained from the quantile strengthening. We exploit the conditional value-at-risk (CVaR) interpretation for \eqref{eq:dr-ccp} described by \citet{xie2018distributionally} to bound key variables in the formulation. Using these, we show that we can further improve our new formulation for \eqref{eq:dr-ccp} by eliminating a significant portion of constraints. Furthermore, the improved bounds we establish on variables allow us to identify a specific substructure in our improved formulation which arises in robust 0-1 programming \cite{atamturk2006robust}. Consequently, we show that we can directly use results from \citet{atamturk2006robust} to describe a class of valid inequalities for our improved formulation.

We then assess the computational impact of the improved formulation and the new valid inequalities on solving \eqref{eq:dr-ccp} on a class of stochastic transportation problems.
We observe that the improved formulation uniformly reduces solution times by at least an order of magnitude for any radius, {number of original decision variables} and number of scenarios. In the difficult small radius regime for $N=100$ scenarios, the original formulation cannot verify optimality within one hour, whereas our formulation solves within seconds. For the most difficult instances with small Wasserstein radius and $N=3000$ scenarios, the improved formulation results in less than $0.8\%$ average optimality gap, whereas the original formulation cannot even find a feasible solution for any of the tested instances.

To the best of our knowledge, our work is the first that examines these connections between traditional SAA, mixing sets, and robust 0-1 programming in the context of distributionally robust chance constraints, and demonstrates the huge computational effectiveness of these approaches for the difficult instances. It is interesting to note that distributionally robust optimization is a paradigm of modeling uncertainty that does not require a complete knowledge of the distribution as in \eqref{eq:saa-ccp}, but is also less conservative than robust optimization that considers the worst-case realizations of uncertain parameters, without any knowledge of their joint distribution. Nevertheless, our analysis and computational results show that the formulations for distributionally robust optimization can be significantly improved by employing its connections to both SAA and robust optimization.

\subsection{Outline} 
In Section \ref{sec:problem}, we provide a formal problem description and elaborate on the connections between the existing MIP models for \eqref{eq:dr-ccp} and the safety set.
In Section \ref{sec:nominal}, we explore the connection between \eqref{eq:saa-ccp} and \eqref{eq:dr-ccp}. We describe the mixing structure in \eqref{eq:saa-ccp} and demonstrate its application to \eqref{eq:dr-ccp}.
In Section \ref{sec:path-ineq}, we employ the CVaR representation of \eqref{eq:dr-ccp} to reveal the substructure from robust 0-1 programming that is hidden in MIP reformulations of \eqref{eq:dr-ccp}. Using results from \citet{atamturk2006robust}, we  provide valid inequalities resulting from this substructure.
In Section \ref{sec:experiments}, we report our computational experience with the improved formulation and the proposed inequalities on a class of stochastic transportation problems. 
	
\section{Problem formulation}\label{sec:problem}

We consider Wasserstein ambiguity sets $\cF_N(\theta)$ defined as the $\theta$-radius Wasserstein ball of distributions on $\bbR^K$ around  the empirical distribution $\bbP_N$. We will use the \emph{1-Wasserstein distance}, based on a norm $\|\cdot\|$, between two distributions $\bbP$ and $\bbP'$. This is defined as follows:
\[ d_W(\bbP,\bbP') := \inf_{\Pi} \left\{ \bbE_{(\vxi,\vxi') \sim \Pi}[\|\vxi - \vxi'\|] : \Pi \text{ has marginal distributions } \bbP, \bbP' \right\}. \]
Then, the Wasserstein ambiguity set is 
\[\cF_N(\theta) := \left\{ \bbP : d_W(\bbP_N,\bbP) \leq \theta\right\}. \] 
Given a decision $\vx \in \cX$ and random realization $\vxi \in \bbR^K$, the \emph{distance from $\vxi$ to the unsafe set} is
\begin{equation}\label{eq:distance}
\dist(\vxi,\cS(\vx)) := \inf_{\vxi' \in \bbR^K} \left\{ \|\vxi - \vxi'\| : \vxi' \not\in \cS(\vx) \right\}.
\end{equation}
Throughout Sections \ref{sec:problem}, \ref{sec:nominal} and \ref{sec:path-ineq}, we assume that the sample $\{\vxi_i\}_{i \in [N]}$, the risk tolerance $\epsilon \in (0,1)$ and the radius $\theta > 0$ are fixed. 
As short-hand notation, we denote the feasible regions of \eqref{eq:saa-ccp} and \eqref{eq:dr-ccp} as follows:
\begin{subequations}\label{eq:feasible-regions}
\begin{align}
\cX_{\SAA}(\cS) &:= \left\{ \vx \in \cX :~ \frac{1}{N} \sum_{i \in [N]} \bm{1}(\vxi_i \not\in \cS(\vx)) \leq \epsilon \right\},\label{eq:saa-region}\\
\cX_{\DR}(\cS) &:= \left\{ \vx \in \cX :~ \sup_{\bbP \in \cF_N(\theta)} \bbP[\vxi \not\in \cS(\vx)] \leq \epsilon \right\}.\label{eq:dr-ccp-region}
\end{align}
\end{subequations}
Note that here the dependence on the safety set  function $\cS$ is made explicit, since the relationship between existing formulations and our new valid inequalities depends on the safety set. 

Using tools from duality theory for Wasserstein distributional robustness \cite{BlanchetMurthy2019,gao2016distributionally}, \citet{chen2018data} and \citet{xie2018distributionally} give an extended formulation for the distributionally robust chance constraint in \eqref{eq:dr-ccp}. Specifically, it was shown in \cite[Theorem 3]{chen2018data} (see also \cite[Proposition 1]{xie2018distributionally}) that when $\cS(\vx)$ is open for each $\vx \in \cX$ and $\theta > 0$, we have
\begin{align}\label{eq:cc-distance-formulation}
\cX_{\DR}(\cS) = \left\{ \vx \in \cX : \begin{aligned}
&\quad \exists\  t \geq 0, \ \vr \geq \bm{0},\\
&\quad \dist(\vxi_i,\cS(\vx)) \geq t - r_i, \ i \in [N],\\
&\quad \epsilon\, t \geq \theta + \frac{1}{N} \sum_{i \in [N]} r_i
\end{aligned} \right\}.
\end{align}
(Note that a similar formulation holds when $\theta = 0$, but we need to make the restriction $t > 0$.) Therefore, the ability to model \eqref{eq:dr-ccp} depends on the ability to model the constraints $\dist(\vxi_i,\cS(\vx)) \geq t - r_i$. For a given set $\cS(\vx)$, let $\intt \cS(\vx)$ denote its interior and $\cl \cS(\vx)$ denote its closure. 
Also, note that \citet[Proposition 3]{gao2016distributionally} show 
\[\sup_{\bbP \in \cF(\theta)} \bbP[\vxi \not\in \intt \cS(\vx)] = \sup_{\bbP \in \cF(\theta)} \bbP[\vxi \not\in \cS(\vx)] = \sup_{\bbP \in \cF(\theta)} \bbP[\vxi \not\in \cl \cS(\vx)],\] and $\dist(\vxi, \intt \cS(\vx)) = \dist(\vxi, \cS(\vx)) = \dist(\vxi, \cl \cS(\vx))$. Therefore, \eqref{eq:cc-distance-formulation} holds regardless of whether $\cS(\vx)$ is open or closed. With this in mind, in what follows, we express $\cS(\vx)$ as an open set, for convenience. 

The classical literature on CCP typically considers two types of safety sets which are  defined by linear inequalities and are known to admit exact MIP reformulations: \emph{individual chance constraints with left-hand side (LHS) uncertainty} and \emph{joint chance constraints with right-hand side (RHS) uncertainty}. In this paper, we consider only joint chance constraints with RHS uncertainty, which have safety set and distance function given by
\begin{subequations}\label{eq:joint-safety}
\begin{align}
\cS(\vx) &:= \left\{ \vxi :~  \vb_p^\top \vxi + d_p - \va_p^\top \vx > 0, \ p \in [P] \right\},  \\
\dist(\vxi,\cS(\vx)) & = \max\left\{ 0,\ \min_{p \in [P]} \frac{\vb_p^\top \vxi + d_p - \va_p^\top \vx}{\|\vb_p\|_*} \right\},
\end{align}
\end{subequations}
for given $\va_p\in \bbR^K,  \vb_p\in \bbR^L$ and $d_p\in \bbR$ for all $p \in [P]$ where $\|\cdot\|_*$  is the dual norm.
\citet[Proposition 2]{chen2018data} show that, in this case, \eqref{eq:dr-ccp} can be reformulated as
\begin{subequations}\label{eq:joint}
\begin{align}
\min\limits_{\vz, \vr, t, \vx} \quad & \vc^\top \vx\label{joint:obj}\\
\text{s.t.}\quad & \vz \in \{0,1\}^N,\ t \geq 0, \ \vr \geq \bm{0},\ \vx \in \mathcal{X},\label{joint:vars}\\
& \epsilon\, t \geq \theta + \frac{1}{N} \sum_{i \in [N]} r_i,\label{joint:conic}\\
& M (1-z_i) \geq t-r_i, \quad i \in [N],\label{joint:bigM1}\\
& \frac{\vb^\top_p {\vxi}_i + d_p -\va^\top_p \vx}{\|\vb_p \|_*}+ M z_i\geq t-r_i, \quad i \in [N],\ p \in [P],\label{joint:bigM2}
\end{align}
\end{subequations}
where $M$ is a sufficiently large positive constant. In particular, for joint chance constraints with RHS uncertainty, we have
\begin{equation}\label{eq:joint-exact}
\cX_{\DR}(\cS) = \left\{ \vx \in \cX : \text{\eqref{joint:vars}--\eqref{joint:bigM2}} \right\} .
\end{equation}
	
\section{Connection with the nominal chance constraint and mixing sets}\label{sec:nominal}

Our first idea to strengthen the MIP reformulations of distributionally robust CCPs originates from the following simple observation between the relation of the empirical probability distribution and the Wasserstein ambiguity set. 

When the radius $\theta$ of the Wasserstein ambiguity set $\cF_N(\theta)$ is 0,~\eqref{eq:dr-ccp} coincides with~\eqref{eq:saa-ccp} since $\cF_N(0)=\left\{\bbP_N\right\}$. In general, as $\cF_N(0)\subseteq \cF_N(\theta)$ for any $\theta\geq0$, \eqref{eq:saa-ccp} is a relaxation of \eqref{eq:dr-ccp}, i.e., we have
\[ \cX_{\DR}(\cS) \subseteq \cX_{\SAA}(\cS). \]
When the safety set is defined as $\cS(\vx) = \left\{ \vxi :~ s(\vx,\vxi) \geq 0\right\}$ for a continuous function $s(\cdot)$, {\citet{ruspp:02} shows} that~\eqref{eq:saa-ccp} can be reformulated as the following MIP:
\begin{subequations}\label{eq:saa-reformulation}
	\begin{align}
	\min\limits_{\vz, \vr, t, \vx} \quad & \vc^\top \vx\label{saa:obj}\\
	\text{s.t.}\quad & \vz \in \{0,1\}^N,\ \vx \in \mathcal{X},\label{saa:vars}\\
	& \sum_{i\in [N]} z_i \leq \lfloor\epsilon N\rfloor, \label{saa:knapsack}\\
	& s(\vx,\vxi_i) + M z_i \geq 0, \quad i \in [N],\label{saa:bigM}
	\end{align}
\end{subequations}
where $M$ is a sufficiently large positive constant. Inequalities~\eqref{saa:knapsack} and~\eqref{saa:bigM} are often referred to as the {\it knapsack (or cardinality) constraint} and the {\it big-$M$ constraints}, respectively. Then, when the safety set $\cS(\vx)$ is given by $\cS(\vx) = \left\{ \vxi :~ s(\vx,\vxi) > 0\right\}$,~\eqref{eq:saa-reformulation} provides a relaxation of~\eqref{eq:dr-ccp}. Therefore, by solving~\eqref{eq:saa-reformulation}, one can provide a lower bound on the optimum value of~\eqref{eq:dr-ccp}. Note that the safety set \eqref{eq:joint-safety} can be written in this form by defining $s(\cdot)$ appropriately. 

The formulation~\eqref{eq:saa-reformulation} for~\eqref{eq:saa-ccp} is well-studied in the literature, and many classes of valid inequalities for the formulation have been developed; see e.g., \cite{abdi2016mixing-knapsack,Kilinc-Karzan2019joint-sumod,kucukyavuz2012mixing,LKL16,liu2018intersection,luedtke2010integer,luedtke2014branch-and-cut,xie2018quantile,zhao2017joint-knapsack}. 
In fact, we develop a more direct connection between  formulation~\eqref{eq:saa-reformulation} and~\eqref{eq:dr-ccp} so that we can apply techniques for solving~\eqref{eq:saa-reformulation} directly to~\eqref{eq:dr-ccp}. It is clear that inequalities of the form~\eqref{saa:knapsack}--\eqref{saa:bigM} can be added to the MIP formulation \eqref{eq:joint} of~\eqref{eq:dr-ccp} to obtain a stronger formulation. In turn, this implies that the reformulation can be further strengthened by the inequalities developed for strengthening~\eqref{eq:saa-reformulation}. If we do this na\"{i}vely, we would add \eqref{saa:knapsack}--\eqref{saa:bigM} to \eqref{eq:joint} with \emph{new} binary variables $\vz' \in \{0,1\}^N$. Our first key result is that \emph{the same binary variables $\vz$} from the MIP formulation \eqref{eq:joint} can be used to define \eqref{saa:knapsack}--\eqref{saa:bigM}. This then means that we can strengthen these formulations \emph{without} adding any additional binary variables.

\subsection{Strengthening the formulation by the nominal chance constraint}

We now verify that the SAA inequalities for the joint chance constraint of~\eqref{eq:dr-ccp-region} can be used to strengthen the formulation~\eqref{eq:joint}. Given $\cS(\vx)$ defined by~\eqref{eq:joint-safety}, consider the following MIP formulation
\begin{subequations}\label{eq:joint-knapsack}
	\begin{align}
	\min\limits_{\vx,\vz, \vr, t} \quad & \vc^\top \vx\label{joint-k:obj}\\
	\text{s.t.}\quad & (\vz,\vr,t,\vx)\ \text{satisfies} \ \eqref{joint:vars}\text{--}\eqref{joint:bigM2},\label{joint-k:basic}\\
	& \sum_{i\in [N]} z_i \leq \lfloor \epsilon N \rfloor, \label{joint-k:knapsack}\\
	& \frac{\vb^\top_p {\vxi}_i + d_p -\va^\top_p \vx}{\|\vb_p \|_*}+ M z_i\geq 0, \quad i \in [N],\ p \in [P],\label{joint-k:bigM3}
	\end{align}
\end{subequations}
where $M$ is a sufficiently large positive constant. (Note that in this formulation, the inequality \eqref{joint-k:knapsack} is equivalent to~\eqref{saa:knapsack}.) We can write constraints~\eqref{joint-k:bigM3} in the form~\eqref{saa:bigM}: individually $s_p(\vx,\vxi_i) := \frac{\vb_p^\top \vxi_i + d_p -\va_p^\top \vx}{\|\vb_p \|_*}$ or jointly $s(\vx,\vxi) := \min\limits_{p\in[P]}\left\{\frac{\vb^\top_p {\vxi}_i + d_p -\va^\top_p \vx}{\|\vb_p \|_*}\right\}$. More importantly,~\eqref{joint-k:knapsack} and~\eqref{joint-k:bigM3} share the same set of binary variables as the other constraints in~\eqref{eq:joint}. We next argue that this MIP formulation~\eqref{eq:joint-knapsack} is an exact reformulation of~\eqref{eq:dr-ccp} with a safety set $\cS(\vx)$ from~\eqref{eq:joint-safety}. 

\begin{theorem}\label{thm:joint-knapsack}
When the safety set $\cS$ is defined as in \eqref{eq:joint-safety}, formulation \eqref{eq:joint-knapsack} is an exact reformulation of~\eqref{eq:dr-ccp}, i.e.,
\begin{equation}
 \cX_{\DR}(\cS) = \left\{ \vx \in \cX : \text{\eqref{joint-k:basic}--\eqref{joint-k:bigM3}} \right\}.
\end{equation}	
\end{theorem}
\begin{proof}
By~\eqref{eq:joint-exact}, it is sufficient to show that 
\[\left\{ \vx \in \cX : \text{\eqref{joint:vars}--\eqref{joint:bigM2}} \right\}=\left\{ \vx \in \cX : \text{\eqref{joint-k:basic}--\eqref{joint-k:bigM3}} \right\}.\] 
By~\eqref{joint-k:basic}, we know that the set on the right-hand side is contained in the set on the left-hand side. 
Let us show that the reverse direction also holds. To this end, take a vector $\vx\in\cX$ satisfying~\eqref{joint:vars}--\eqref{joint:bigM2} with some $\vz,\vr,t$. For ease of notation, we define $s(\vx,\vxi_i):=\min\limits_{p\in[P]}\left\{\frac{\vb^\top_p {\vxi}_i + d_p -\va^\top_p \vx}{\|\vb_p \|_*}\right\}$. We claim that $(\vx,\bar\vz,\vr,t)$ satisfies~\eqref{joint-k:basic}--\eqref{joint-k:bigM3} where $\bar\vz\in\{0,1\}^N$ is a vector such that $\bar z_i=1$ if and only if $s(\vx,\vxi_i)<0$
for all $i\in[N]$. Since $M$ is sufficiently large so that $s(\vx,\vxi_i)+M\geq 0$, it is clear that $\vx,\bar\vz$ satisfy~\eqref{joint-k:bigM3}. Next, we observe that $(\vx,\bar\vz,\vr,t)$ satisfies~\eqref{joint:bigM1} and~\eqref{joint:bigM2}, which can be equivalently rewritten as \[\min\left\{s(\vx,\vxi_i)+Mz_i,\, M(1-z_i)\right\}\geq t-r_i, \quad i \in [N].\]
By the choice of $\bar\vz$ and $M$, we have $\min\left\{s(\vx,\vxi_i)+M\bar z_i,\, M(1-\bar z_i)\right\}$ is equal to 0 if $s(\vx,\vxi_i)<0$, and it is equal to $s(\vx,\vxi_i)$ otherwise. Hence,
\[\min\left\{s(\vx,\vxi_i)+M\bar z_i,\, M(1-\bar z_i)\right\} %
\geq \min\left\{s(\vx,\vxi_i)+Mz_i,\, M(1-z_i)\right\}\]
for any $z_i\in\{0,1\}$. This implies that $(\vx,\bar\vz,\vr,t)$ satisfies~\eqref{joint:bigM1} and~\eqref{joint:bigM2} as they are satisfied already by $(\vx,\vz,\vr,t)$. To finish the proof, it remains to show that $\bar\vz$ satisfies~\eqref{joint-k:knapsack}. Since~$\theta>0$, we obtain from~\eqref{joint:conic} that $t>0$. This implies that $\frac{r_i}{t}\geq0$. We claim that $\frac{r_i}{t}\geq \bar z_i$ for all $i\in[N]$. As $\frac{r_i}{t}\geq 0$, we have $\frac{r_i}{t}\geq \bar z_i$ holds when $\bar z_i=0$. When $\bar z_i=1$, by rearranging~\eqref{joint:bigM1} we get $\frac{r_i}{t}\geq 1=\bar z_i$. Since $\frac{r_i}{t}\geq \bar z_i$ for all $i\in[N]$, it follows from~\eqref{joint:conic} that $\epsilon N\geq \sum_{i\in[N]}\frac{r_i}{t}\geq\sum_{i\in[N]}\bar z_i$, implying in turn that $\sum_{i\in[N]}\bar z_i\leq \lfloor\epsilon N\rfloor$, as required.
\ifx\flagJournal\true \qed \fi
\end{proof}

\begin{remark}\label{rem:bigM-joint}
\citet{chen2018data} argue that there is a finite value of $M$ for the validity of~\eqref{eq:joint}. Essentially, we need to choose an appropriate value of $M$ so that~\eqref{joint:bigM1}--\eqref{joint:bigM2} correctly represent the constraint $\dist(\vxi_i,\cS(\vx)) \geq t - r_i$. When $\vb^\top_p {\vxi}_i + d_p -\va^\top_p \vx<0$ for some $p\in[P]$, we have $\dist(\vxi_i,\cS(\vx))=0$ by~\eqref{eq:joint-safety} and thus $\dist(\vxi_i,\cS(\vx)) \geq t - r_i$ becomes $0\geq t-r_i$. As long as $\frac{\vb^\top_p {\vxi}_i + d_p -\va^\top_p \vx}{\|\vb_p \|_*} + M\geq 0$ holds for all $\vx\in\cX$, \eqref{joint:bigM1}--\eqref{joint:bigM2} can capture this situation by setting $z_i=1$. Similarly,~\eqref{joint:bigM1}--\eqref{joint:bigM2} can represent the case when $\vb^\top_p {\vxi}_i + d_p -\va^\top_p \vx\geq 0$ for all $p\in[P]$ if $M\geq \vb^\top_p {\vxi}_i + d_p -\va^\top_p \vx$ for all $\vx\in\cX$, $p\in[P]$. So, setting
\begin{equation}\label{eq:bigMvalue-joint}
M:=\max_{\vx\in\cX,\ p\in[P]}\left\{\frac{\left|\vb^\top_p {\vxi}_i + d_p -\va^\top_p \vx\right|}{\|\vb_p \|_*}\right\}
\end{equation}
ensures that the constraints~\eqref{joint-k:bigM3} are indeed valid.	
\ifx\flagJournal\true \epr \fi
\end{remark}

\subsection{Mixing substructure}\label{sec:mixing}

Although~\eqref{eq:joint-knapsack} is already stronger than~\eqref{eq:joint} due to the additional constraints~\eqref{joint-k:bigM3}--\eqref{joint-k:knapsack}, we can further strengthen this formulation by adding more valid inequalities originating from these constraints. To this end, for any fixed $p$, the constraints~\eqref{joint-k:bigM3}--\eqref{joint-k:knapsack} give rise to the following substructure:
\begin{equation}
Q_p:=\left\{ (\vx,\vz)\in \cX\times\{0,1\}^N : \begin{aligned}
\ &s_p(\vx,\vxi_i) + M z_i \geq 0, \quad i \in [N],\\
& \sum_{i\in [N]} z_i \leq \lfloor\epsilon N\rfloor
\end{aligned} \right\}.\label{eq:mixingset}
\end{equation}
So, one can generate valid inequalities for the formulation~\eqref{eq:joint-knapsack} by finding inequalities of the form $\vmu^\top \vx + \vpi^\top \vz\geq\beta$ that are valid for the mixed-integer set $Q_p$ in~\eqref{eq:mixingset}. \citet{luedtke2010integer} and \citet{luedtke2014branch-and-cut} introduce a procedure of generating inequalities that are valid for $Q_p$. In order to make our paper self-contained, we next explain this procedure.

Given a fixed linear function $\vmu^\top\vx$, we solve the following single scenario subproblem for each scenario $i\in[N]$
\begin{equation}
\bar h_i(\vmu):=\min\left\{\vmu^\top\vx:\ s_p(\vx,\vxi_i)\geq0, \ \vx\in \bar\cX\right\},\label{eq:single-subproblem}
\end{equation}
where $\cX \subseteq \bar \cX \subseteq \bbR^L$. Then $\vmu^\top\vx\geq \bar h_i(\vmu)$ holds for $(\vx,\vz)\in Q_p$ with $z_i=0$. Having computed the values $\bar h_i(\vmu)$ for $i\in[N]$, we sort them in non-decreasing order. Without loss of generality, we may assume that
\[
\bar h_N(\vmu)\geq\bar  h_{N-1}(\vmu)\geq \cdots\geq  \bar h_1(\vmu).
\]
For ease of notation, let 
\[k:=\lfloor\epsilon N\rfloor.\]
Notice that because $\sum_{i\in [N]} z_i \leq k$ is also enforced in $Q_{p}$, by the pigeonhole principle there must exist $i \in \{N-k,N-k+1,\ldots,N\}$ with $z_i=0$. In turn, this implies that $\vmu^\top\vx \geq \bar h_{N-k}(\vmu)$ because $\bar h_i(\vmu)\geq \bar h_{N-k}(\vmu)$ for all $i\geq N-k$. In summary, we have just argued that $\vmu^\top\vx\geq \bar h_i(\vmu)$ holds if $z_i=0$ and that $\vmu^\top\vx\geq \bar  h_{N-k}(\vmu)$ is satisfied always, in particular, when $z_i=1$ for $i\in [N]$. Equivalently,
\begin{equation}
\vmu^\top\vx + \left(\bar h_i(\vmu)-\bar h_{N-k}(\vmu)\right)z_i\geq\bar  h_i(\vmu)\label{eq:mixing-base}
\end{equation}
is valid.
In fact, inequalities~\eqref{eq:mixing-base} for $i\leq N-k$ are redundant because $\vmu^\top\vx\geq \bar  h_i(\vmu)$ is implied by $\vmu^\top\vx\geq \bar  h_{N-k}(\vmu)$ if $i\leq N-k$. 
Because inequalities~\eqref{eq:mixing-base} share a common linear function $\vmu^\top\vx$ but each one has a distinct integer variable, the {\it mixing procedure} of~\citet{gunluk2001mixing}, can be applied to obtain stronger inequalities. For any $J=\{j_1,\ldots,j_\ell\}$ with $N\geq j_1\geq\cdots\geq j_\ell\geq N-k+1$, the {\it mixing inequality} derived from $J$ and~\eqref{eq:mixing-base} is
\begin{equation}
\vmu^\top\vx + \sum_{i\in[\ell]}\left(\bar h_{j_i}(\vmu)-\bar h_{j_{i+1}}(\vmu)\right)z_{j_i}\geq \bar h_{j_1}(\vmu),\label{eq:mixing}
\end{equation}
where $j_{\ell+1}:=N-k$. The mixing inequalities of the form~\eqref{eq:mixing} are equivalent to the {\it star inequalities} by~\citet{atamturk2000mixed}. The inequalities \eqref{eq:mixing} are the strongest possible ones that can be generated from~\eqref{eq:mixing-base} in that the convex hull of solutions $(\vx,\vz)\in\bbR^L\times\{0,1\}^N$ satisfying~\eqref{eq:mixing-base} is described by~\eqref{eq:mixing}~\cite{atamturk2000mixed,gunluk2001mixing,Kilinc-Karzan2019joint-sumod}.

Consider any $p\in[P]$. For $s_p(\vx,\vxi_i) := \frac{1}{\|\vb_p \|_*}\left(\vb^\top_p {\vxi}_i + d_p -\va^\top_p \vx\right)$ as in~\eqref{eq:joint-knapsack}, we can take $-\frac{1}{\|\vb_p \|_*}\va_p$ for $\vmu$. For this choice of $\vmu$, the value of $\bar h_i(\vmu)$ from the single scenario subproblem~\eqref{eq:single-subproblem} with $\bar\cX=\R^L$ is precisely 
\[
\min_x\left\{-\frac{\va^\top_p}{\|\vb_p \|_*}\vx:\ \frac{1}{\|\vb_p \|_*}\left(\vb^\top_p {\vxi}_i + d_p -\va^\top_p \vx\right)\geq0\right\}=-\frac{\vb^\top_p {\vxi}_i + d_p}{\|\vb_p \|_*}.
\]
So, assuming $-\vb^\top_p {\vxi}_N\geq\cdots\geq -\vb^\top_p {\vxi}_1$ and letting \[q_p:=-\vb_p^\top\vxi_{N-k}\] ($q_p$ is the $(k+1)$-th largest value), the inequalities $\vmu^\top\vx\geq \bar h_{N-k}(\vmu)$ and~\eqref{eq:mixing-base} for $\vmu=-\frac{1}{\|\vb_p \|_*}\va_p$ correspond to
\begin{subequations}\label{eq:mixing-base1}
\begin{align}
&\frac{-q_p + d_p -\va^\top_p \vx}{\|\vb_p \|_*} \geq 0,&\label{eq:mixing-base-quantile}\\
&\frac{\vb^\top_p \vxi_i + d_p -\va^\top_p \vx}{\|\vb_p \|_*} + \frac{-\vb^\top_p {\vxi}_i-q_p}{\|\vb_p \|_*}z_i \geq 0,& i\in[N].\label{eq:mixing-base-rhs}
\end{align}
\end{subequations}
Moreover, the mixing inequalities~\eqref{eq:mixing} obtained from~\eqref{eq:mixing-base} have the following form:
\begin{equation}\label{eq:mixing-rhs}
\frac{\vb^\top_p {\vxi}_{j_1} + d_p -\va^\top_p \vx}{\|\vb_p \|_*} + \sum_{i\in[\ell]}\frac{-\vb^\top_p {\vxi}_{j_i} + \vb^\top_p {\vxi}_{j_{i+1}}}{\|\vb_p \|_*}  z_{j_i} \geq 0
\end{equation}
where $N\geq j_1\geq\cdots\geq j_\ell\geq N-k+1$ and $j_{\ell+1}:=N-k$.

The number of mixing inequalities~\eqref{eq:mixing-rhs} is exponential, but they can be separated in $O(N\log N)$ time~\cite{gunluk2001mixing,Kilinc-Karzan2019joint-sumod}.

Notice that the inequalities~\eqref{eq:mixing-base}--\eqref{eq:mixing-rhs} are \emph{big-$M$-free}; the coefficients of the binary variables and the right-hand sides depend only on~$s_p(\vx,\vxi_i)$. This is important for practical purposes. In particular, when %
$M$ is larger than $\frac{-\vb^\top_p {\vxi}_i -q_p}{\|\vb_p \|_*}$, the constraints~\eqref{joint-k:bigM3} are dominated by the inequalities~\eqref{eq:mixing-base-rhs}, and thus, by~\eqref{eq:mixing-rhs}.

\subsection{Reducing big-$M$ values}\label{sec:M-reduction}

In Section~\ref{sec:mixing}, we argued that letting $q_p$ be the $(k+1)$-th largest value amongst $\{-\vb_p^\top \vxi_i\}_{i \in [N]}$, the inequalities~\eqref{eq:mixing-base1} are valid for~\eqref{eq:joint-knapsack}, and based on this we generate the mixing inequalities~\eqref{eq:mixing-rhs}. In fact, we can replace the big-$M$ in \eqref{joint:bigM2} with its strengthened version from \eqref{eq:mixing-base-rhs} to obtain a new formulation:
\begin{subequations}\label{eq:joint-k-reduced}
	\begin{align}
	\min\limits_{\vz, \vr, t, \vx} \quad & \vc^\top \vx\\
	\text{s.t.}\quad & (\vz,\vr,t,\vx)\ \text{satisfies} \ \eqref{joint:vars}\text{--}\eqref{joint:bigM1}\ \text{and} \ \eqref{joint-k:knapsack},\label{joint-k:basic-reduced}\\
	& \frac{\vb^\top_p {\vxi}_i + d_p -\va^\top_p \vx}{\|\vb_p \|_*}+ \frac{-\vb^\top_p {\vxi}_i - q_p}{\|\vb_p \|_*} z_i\geq t-r_i, \quad i \in [N],\ p \in [P]\label{joint-k:bigM2-reduced}.
	\end{align}
\end{subequations}
\begin{theorem}\label{thm:improved-formulation}
Formulation	\eqref{eq:joint-k-reduced} is an exact reformulation of~\eqref{eq:dr-ccp} where the safety set is given by~\eqref{eq:joint-safety}.
\end{theorem}
\begin{proof}
	By Theorem~\ref{thm:joint-knapsack},~\eqref{eq:joint} with~\eqref{joint-k:knapsack} is an exact reformulation. Hence, we need to argue that reducing the big-$M$ value in~\eqref{joint:bigM2} to obtain~\eqref{joint-k:bigM2-reduced} keeps the formulation valid. To this end, it suffices to argue that constraints~\eqref{joint:bigM1} and~\eqref{joint-k:bigM2-reduced} correctly represent $\dist(\vxi_i,\cS(\vx)) \geq t - r_i$ for $i\in[N]$. 
	Let $i\in[N]$. If $\min\limits_{p \in [P]} \left\{\frac{\vb_p^\top \vxi_i + d_p - \va_p^\top \vx}{\|\vb_p\|_*}\right\}\geq0$, we can set $z_i=0$ and we obtain $\dist(\vxi_i,\cS(\vx))=\min\limits_{p \in [P]} \left\{\frac{\vb_p^\top \vxi_i + d_p - \va_p^\top \vx}{\|\vb_p\|_*}\right\}$. In this case,~\eqref{joint:bigM1} becomes redundant and~\eqref{joint-k:bigM2-reduced} represents $\dist(\vxi_i,\cS(\vx)) \geq t - r_i$. On the other hand, if we set $z_i = 1$, then \eqref{joint:bigM1} is $0 \geq t - r_i$ and \eqref{joint-k:bigM2-reduced} becomes redundant. However, $\dist(\vxi_i,\cS(\vx)) \geq t - r_i$ is less restrictive on $(\vx,\vr,t)$ than $0 \geq t-r_i$, so at optimality, there is at least one solution such that $z_i = 0$ whenever $\min\limits_{p \in [P]} \left\{\frac{\vb_p^\top \vxi_i + d_p - \va_p^\top \vx}{\|\vb_p\|_*}\right\}\geq0$, hence $\dist(\vxi_i,\cS(\vx)) \geq t - r_i$ is correctly represented.
	
	{If $\min\limits_{p \in [P]} \left\{\frac{\vb_p^\top \vxi_i + d_p - \va_p^\top \vx}{\|\vb_p\|_*}\right\}<0$, then $\dist(\vxi_i,\cS(\vx))=0$ and we must have $z_i=1$ by~\eqref{eq:mixing-base-rhs}.
	Since $z_i=1$,~\eqref{joint-k:bigM2-reduced} and \eqref{joint:bigM1} respectively become
	\begin{align*}
	t - r_i &\leq \frac{\vb_p^\top \vxi_i + d_p - \va_p^\top \vx}{\|\vb_p\|_*}+ \frac{-\vb^\top_p {\vxi}_i -q_p}{\|\vb_p \|_*}=\frac{-q_p + d_p - \va_p^\top \vx}{\|\vb_p\|_*}\\
	t - r_i &\leq 0.
	\end{align*}
	By \eqref{eq:mixing-base-quantile} we know that $\frac{-q_p + d_p - \va_p^\top \vx}{\|\vb_p\|_*} \geq 0$ so the first inequality is weaker than the second, hence we have $\dist(\vxi_i,\cS(\vx))=0 \geq t-r_i$ as required.
	}
	\ifx\flagJournal\true \qed \fi
\end{proof}

In contrast to \eqref{joint:bigM2}, the inequality \eqref{joint-k:bigM2-reduced} is \emph{big-$M$-free}. During our computational study summarized in Section \ref{sec:experiments}, we observed that the coefficients of the variables $\vz$ in \eqref{joint-k:bigM2-reduced} are significantly smaller than the big-$M$ value computed from Remark \ref{rem:bigM-joint}.

\section{Improved formulation and valid inequalities from robust 0-1 programming}\label{sec:path-ineq}

In this section, we study, in closer detail, the set of constraints from~\eqref{joint-k:bigM2-reduced}. By appealing to the conditional value-at-risk interpretation for \eqref{eq:dr-ccp}, we prove a bound on $t$. Consequently, we uncover another hidden structure related to robust 0-1 programming with a budget uncertainty set. This allows us to derive a new class of valid inequalities.

For a given $p\in [P]$, recall the definition of $q_p$ being the $(k+1)$-th largest value amongst $\left\{-\vb_p^\top \vxi_i\right\}_{i \in [N]}$ where $k:= \lfloor \epsilon N \rfloor$. We also define $h_{i,p} := \frac{-\vb_p^\top \vxi_i -q_p}{\|\vb_p\|_*}$ and $u_p := \frac{-q_p+d_p-\va_p^\top \vx}{\|\vb_p\|_*} - t$. The constraints~\eqref{joint-k:bigM2-reduced} can be cast as the mixed-integer set
\[R_p :=\left\{ (u_p,\vr,\vz) \in \bbR \times \bbR_+^N \times \{0,1\}^N: \ 
u_p + r_i \geq h_{i,p} (1-z_i),\ i \in [N]\right\}.\]
In fact, a similar mixed-integer set has been studied in the context of robust 0-1 programming by~\citet{atamturk2006robust}:
\[
R^+ := \left\{ (u,\vr,\vz)\in\bbR_+\times\bbR_+^N \times\{0,1\}^N: \  u + r_i \geq h_i (1-z_i),\ i \in [N] \right\},
\]
where $h_1,\ldots, h_N$ are assumed to all be positive. The main difference between $R_p$ and $R^+$ is that $u_p$ is unrestricted in $R_p$ but is non-negative in $R^+$, and some $h_{i,p}$ may be negative or zero in $R_p$, but all $h_i$ are positive in $R^+$. We refer to the set $R^+$ as the \emph{robust 0-1 set}, since it originates from the robust counterpart of a 0-1 program whose objective vector belongs to a ``budget uncertainty set". In \citet{atamturk2006robust}, the binary variables in $R^+$ correspond to the original decision variables of the 0-1 program.

The fact that $R_p$ is so similar to $R^+$ may seem surprising at first, yet is less so if we consider the conditional value-at-risk (CVaR) interpretation for the Wasserstein robust chance constraint. More precisely, \citet[Corollary 1]{xie2018distributionally} gives the following alternate version of \eqref{eq:cc-distance-formulation}:
\begin{subequations}
\begin{align}
\cX_{\DR}(\cS)
&= \left\{ \vx \in \cX :~ \frac{\theta}{\epsilon} + \CVaR_{1-\epsilon}\left( -\dist(\vxi,\cS(\vx));\bbP_N \right) \leq 0 \right\}\label{eq:cc-distance-cvar}\\
&= \left\{ \vx \in \cX :~ \frac{\theta}{\epsilon} + \max_{y \in B} \frac{1}{\epsilon N} \sum_{i \in [N]} (-\dist(\vxi_i,\cS(\vx))) y_i \leq 0 \right\}\label{eq:cc-distance-cvar-dual}\\
\text{where }  &\CVaR_{1-\epsilon}(v(\vxi);\bbP_N) := \min_t \left\{t + \frac{1}{\epsilon N} \sum_{i \in [N]} \max\{0,v(\vxi_i) - t\} \right\},\label{eq:cvar-def}\\
&B := \left\{ \vy :~ \frac{1}{N} \sum_{i \in [N]} y_i = \epsilon, \ \bm{0} \leq \vy \leq \bm{1} \right\}.\label{eq:budget-uncertainty}
\end{align}
\end{subequations}
Note that \eqref{eq:cc-distance-cvar-dual} comes from the dual interpretation of the CVaR, see e.g., \citep[(3)]{Noyan2018tutorial}. In particular, $B$ has exactly the same structure as the budget uncertainty set studied in \citet{atamturk2006robust}.

We now further exploit the CVaR interpretation \eqref{eq:cc-distance-cvar} of \eqref{eq:cc-distance-formulation} so that we can cast $R_p$ \emph{exactly} in the same form as $R^+$. Specifically, we show that $u_p \geq 0$ is a valid inequality for \eqref{eq:joint-k-reduced}, which will then allow us to eliminate constraints with negative $h_{i,p}$. Consequently, this allows us to directly apply the valid inequalities derived for $R^+$ from \citet{atamturk2006robust} to $R_p$. We first provide a bound on the value of the $t$-variable in \eqref{eq:cc-distance-formulation} (which then translates to a bound on the $t$-variables for all subsequent formulations).
\begin{lemma}\label{lemma:t-bound}
Let $k:=\lfloor \epsilon N \rfloor$ and fix any $\vx \in \cX_{\DR}(\cS)$. Then there exists $(\vr,t)$ such that $t$ is equal to the $(k+1)$-th smallest value amongst $\{\dist(\vxi_i,\cS(\vx))\}_{i \in [N]}$ and the constraints of \eqref{eq:cc-distance-formulation} are satisfied.
\end{lemma}
\begin{proof}

{
Recall that the CVaR of a random variable $v(\vxi)$ from \eqref{eq:cvar-def}, where $\vxi \sim \bbP_N$, has two equivalent primal and dual optimization representations:
\begin{subequations}\label{eq:cvar-primal-dual}
\begin{align}
\CVaR_{1-\epsilon}(v(\vxi);\bbP_N) &= \max_{y \in B} \frac{1}{\epsilon N} \sum_{i \in [N]} v(\vxi_i) y_i\\
&= \min_{\vr,t'} \left\{t' + \frac{1}{\epsilon N} \sum_{i \in [N]} r_i : r_i \geq v(\vxi_i) - t', i \in [N], \vr \geq \mathbf 0 \right\}, \label{eq:cvar-lp}
\end{align}
\end{subequations}
where $B$ is defined in \eqref{eq:budget-uncertainty}. Without loss of generality, assume that we have an ordering $v(\vxi_1) \geq \dots \geq v(\vxi_N)$. It is easy to check that a primal-dual optimal pair for \eqref{eq:cvar-primal-dual} is given by 
\begin{align*}
y_i &= \begin{cases}
1, &i=1,\ldots,k\\
\epsilon N - k, &i=k+1\\
0, &i > k+1
\end{cases}\\
t' &= v(\vxi_{k+1})\\
r_i &= \max\left\{ 0, v(\vxi_i) - t'\right\},
\end{align*}
as $\vy \in B$ and $\vr \geq \bm{0}$, and
\[ \frac{1}{\epsilon N} \sum_{i \in [N]} v(\vxi_i) y_i = t' + \frac{1}{\epsilon N} \sum_{i \in [N]} r_i. \]
}

Now take $v(\vxi_i) = -\dist(\vxi_i,\cS(\vx))$ for each $i \in [N]$, and recall that we assume $v(\vxi_1)=-\dist(\vxi_1,\cS(\vx)) \geq \cdots \geq v(\vxi_N)=-\dist(\vxi_N,\cS(\vx))$. {Let $(\vr,t')$ be the optimal solution to the CVaR formulation \eqref{eq:cvar-lp} specified above with $t' = v(\vxi_{k+1}) = -\dist(\vxi_{k+1},\cS(\vx))$, $r_i=\max\{0, v(\vxi_i) - t'\}$ for all $i \in [N]$.} Since $\vx \in \cX_{\DR}(\cS)$, we know from \eqref{eq:cc-distance-cvar} that
\[ \frac{\theta}{\epsilon} + \CVaR_{1-\epsilon}\left( -\dist(\vxi,\cS(\vx));\bbP_N \right) \leq 0 \implies \theta + \frac{1}{N} \sum_{i \in [N]} r_i \leq -\epsilon t'. \]
Now take $t = -t' = \dist(\vxi_{k+1},\cS(\vx)) \geq 0$, and notice that {$r_i = \max\left\{ 0, v(\vxi_i) - t'\right\} = \max\left\{ 0, t - \dist(\vxi_i,\cS(\vx))\right\}$ for all $i \in [N]$}, so the constraints in \eqref{eq:cc-distance-formulation} are satisfied.
\ifx\flagJournal\true \qed \fi
\end{proof}

The main result for this section is to show that $u_p \geq 0$ is valid for \eqref{eq:joint-k-reduced} for each $p \in [P]$.
\begin{proposition}\label{prop:u_p-bound}
Suppose that $\theta > 0$. Consider an arbitrary $\vx \in \cX_{\DR}(\cS)$. There exists $(\vr,t,\vz)$ such that $(\vx,\vr,t,\vz)$ satisfies \eqref{joint-k:basic-reduced}--\eqref{joint-k:bigM2-reduced} and that for every $p \in [P]$,
\[ u_p= \frac{-q_p+d_p-\va_p^\top \vx}{\|\vb_p\|_*} - t \geq 0. \]
\end{proposition}
\begin{proof}
For convenience,  for each $p \in [P]$, denote $g_{i,p}(\vx) := \frac{\vb_p^\top \vxi_i+d_p-\va_p^\top \vx}{\|\vb_p\|_*}$ and $g_p^*(\vx) := \frac{-q_p+d_p-\va_p^\top \vx}{\|\vb_p\|_*}$. With these definitions, the distance function from \eqref{eq:joint-safety} is
\[ \dist(\vxi_i,\cS(\vx)) = \max\left\{ 0, \min_{p \in [P]} g_{i,p}(\vx) \right\}. \]
Without loss of generality, assume that $\dist(\vxi_1,\cS(\vx)) \leq \cdots \leq \dist(\vxi_N,\cS(\vx))$, and denote $d^*(\vx) := \dist(\vxi_{k+1},\cS(\vx))$ to be the $(k+1)$-smallest distance value.

By Lemma \ref{lemma:t-bound}, take $(\vr,t)$ to be a solution that satisfies $t = d^*(\vx)$, {$r_i = \max\left\{ 0, t - \dist(\vxi_i,\cS(\vx)) \right\}$}, and set $z_i = 1$ when $\min_{p \in [P]} g_{i,p}(\vx) < 0$, otherwise $z_i = 0$ for each $i \in [N]$.
{It is straightforward to check that $(\vx,\vr,t,\vz)$ satisfies \eqref{joint-k:basic-reduced}, so we focus on \eqref{joint-k:bigM2-reduced}. If $z_i = 0$, then since $\dist(\vxi_i,\cS(\vx)) \geq t - r_i$ and $\dist(\vxi_i,\cS(\vx))=\min_{p \in [P]} g_{i,p}(\vx)$, \eqref{joint-k:bigM2-reduced} holds. If $z_i=1$, then $\dist(\vxi_i,\cS(\vx))=0$, and $r_i=t$, so $t-r_i=0$, hence \eqref{joint-k:bigM2-reduced} reduces to \eqref{eq:mixing-base-quantile}.}

It remains to show that for every $p \in [P]$, $0\le u_p=g_p^*(x)-t=g_p^*(x)-d^*(x)$, i.e., $d^*(\vx) \leq g_p^*(\vx)$ holds. Note that by definition of $-q_p$, $g_p^*$ is the $(k+1)$-th smallest value amongst $\{g_{i,p}(\vx)\}_{i \in [N]}$. Focusing on the definition of $d^*(\vx) := \dist(\vxi_{k+1},\cS(\vx))$, we consider two cases.

If $\min_{p \in [P]} g_{k+1,p}(\vx) \leq 0$, then $d^*(\vx) = t = 0$. But, then we cannot have $\epsilon t \geq \theta + \frac{1}{N} \sum_{i \in [N]} r_i$, since $\theta > 0$ and $\vr\ge\mathbf 0$. Thus, we cannot have $\min_{p \in [P]} g_{k+1,p}(\vx) \leq 0$.

Now consider $\min_{p \in [P]} g_{k+1,p}(\vx) > 0$. Then, for any $i \geq k+1$ and $p \in [P]$, we have 
\begin{align*}
0 <  \min_{p \in [P]} g_{k+1,p}(\vx) =d^*(\vx) \le \dist(\vxi_i,\cS(\vx)) = \min_{p' \in [P]} g_{i,p'}(\vx) \leq g_{i,p}(\vx),
\end{align*}
where the second inequality follows from the fact that $i\ge k+1$ and the assumption that $\dist(\vxi_1,\cS(\vx)) \leq \cdots \leq \dist(\vxi_N,\cS(\vx))$, and the last equation follows from the fact that $0 < \dist(\vxi_i,\cS(\vx))$.
Based on this relation, then there are at least $N-k$ indices $i$ such that $g_{i,p}(\vx) \geq d^*(\vx)$, and thus we must have $g_p^*(\vx) \geq d^*(\vx)$. This completes the proof.
\ifx\flagJournal\true \qed \fi
\end{proof}

We can now explicitly impose the constraint $u_p \geq 0$ into the formulation for \eqref{eq:dr-ccp}. When we do this, the constraint $u_p + r_i \geq h_{i,p} (1-z_i)$ corresponding to any indices $i,p$ such that $h_{i,p} \leq 0$ becomes redundant since the left-hand side of this constraint is non-negative, and its right-hand side is non-positive. Based on this, in our improved formulation, we define the following index sets:
\begin{align*}
[N]_p := \left\{ i \in [N] : -\vb_p^\top \vxi_i > q_p \right\}, \quad p \in [P].
\end{align*}
Then, $h_{i,p} > 0$ if and only if $i \in [N]_p$. Our proposed formulation is as follows: 
\begin{subequations}\label{eq:joint-k-compact}
	\begin{align}
	\min\limits_{\vz, \vr, t, \vx} \quad & \vc^\top \vx\\
	\text{s.t.}\quad & (\vz,\vr,t,\vx)\ \text{satisfies} \ \eqref{joint:vars}\text{--}\eqref{joint:bigM1}\ \text{and} \ \eqref{joint-k:knapsack},\label{joint-k:basic-compact}\\
	&\frac{\vb^\top_p {\vxi}_i + d_p -\va^\top_p \vx}{\|\vb_p \|_*}+ \frac{-\vb^\top_p {\vxi}_i - q_p}{\|\vb_p \|_*} z_i\geq t-r_i, \quad i \in [N]_p,\ p \in [P],\label{joint-k:bigM2-compact}\\
	&\frac{-q_p + d_p -\va^\top_p \vx}{\|\vb_p \|_*} \geq t, \quad p \in [P].\label{joint-k:nonnegative-compact}
	\end{align}
\end{subequations}
\begin{theorem}\label{thm:improved-formulation-compact}
	Formulation	\eqref{eq:joint-k-compact} is an exact reformulation of~\eqref{eq:dr-ccp} where the safety set is given by~\eqref{eq:joint-safety}.
\end{theorem}
\begin{proof}
The correctness of \eqref{eq:joint-k-compact} follows from the above discussion.
\ifx\flagJournal\true \qed \fi
\end{proof}

\begin{remark}
The size of each index set $[N]_p$ is at most $k = \lfloor \epsilon N \rfloor$, so \eqref{eq:joint-k-compact} reduces the number of constraints of \eqref{eq:joint-k-reduced} by at least $((1-\epsilon) N - 1) P$. This is particularly significant when $\epsilon$ is small (e.g., $0.1$) and $N,P$ are large.
\ifx\flagJournal\true \epr \fi
\end{remark}
Note that  similar bounding and scenario elimination strategies using the VaR interpretation are shown to be effective in multivariate CVaR-constrained optimization and risk-averse Markov decision processes \citep{kuccukyavuz2016cut,Liu2017,Noyan2019,Merakli2019}.

Importantly, via the constraints \eqref{joint-k:bigM2-compact} and \eqref{joint-k:nonnegative-compact}, we can re-define the mixed-integer set $R_p$ to have the exact same structure as $R^+$:
\[R_p :=\left\{ (u_p,\vr,\vz)\in\bbR_+ \times\bbR_+^N \times\{0,1\}^N:\ u_p + r_i \geq h_{i,p} (1-z_i),\ i \in [N]_p\right\},\]
where it is understood that $u_p = \frac{-q_p+d_p-\va_p^\top \vx}{\|\vb_p\|_*} - t$. \citet[Theorem 1]{atamturk2006robust} proposes a class of valid inequalities for $R^+$, which we adapt to obtain valid inequalities for $R_p$. To do this, we let $N_p$ be the size of each $[N]_p$ and define an ordering $[N]_p = \left\{ (j,p) \in \bbN : j \in [N_p] \right\}$ as follows: 
\[ h_{(N_p,p),p} \geq \cdots \geq h_{(1,p),p}. \]

\begin{proposition}
	For any $p \in [P]$ and $J=\{j_1,\ldots, j_m\}$ satisfying $m \geq 1$, $N_p \geq j_1\geq\cdots\geq j_m\geq 1$, the following inequality is valid for \eqref{eq:joint-k-compact}:
	\begin{equation}\label{eq:path}
	\frac{-q_p+d_p-\va_p^\top \vx}{\|\vb_p\|_*} - t + \sum_{i\in[m]}r_{(j_i,p)} \geq \sum_{i\in[m]} (h_{(j_i,p),p}-h_{(j_{i+1},p),p})(1-z_{(j_i,p)})
	\end{equation}
	where $h_{(j_{m+1},p),p}:=0$.
\end{proposition}
\begin{proof}
	This follows immediately from \citet[Theorem 1]{atamturk2006robust}.
	\ifx\flagJournal\true \qed \fi
\end{proof}
We refer to the inequalities~\eqref{eq:path} as the \emph{path inequalities}.

\begin{remark}\label{rem:sep-pos-path}
\citet{atamturk2006robust} gives an $O(N_p^2) = O(\lfloor\epsilon N\rfloor^2)$-time separation algorithm for~\eqref{eq:path}, which is based on finding a shortest path in an acyclic graph. \citet[Theorems 1 and 2]{atamturk2006robust} also proves that inequalities~\eqref{eq:path} are sufficient to describe the convex hull of $R^+$ and are facet-defining.
	\ifx\flagJournal\true \epr \fi
\end{remark}

\section{Computational Study}\label{sec:experiments}

In this section, we assess the numerical performance of our improved formulation \eqref{eq:joint-k-compact} of \eqref{eq:dr-ccp} and valid inequalities from Sections~\ref{sec:nominal} and~\ref{sec:path-ineq}. 

All experiments are conducted on an Intel Core i5 3GHz processor with 6 cores and 32GB memory.  Each experiment was in single-core mode, and five experiments were run in parallel. For each model, we set the CPLEX time limit to be 3600 seconds.

CPLEX 12.9 is used as the MIP solver. Valid inequalities from Sections~\ref{sec:nominal} and~\ref{sec:path-ineq} are separated and added via the CPLEX user-cut callback feature.  Since using a user-cut callback function is known to affect various internal  CPLEX dynamics (such as dynamic search, aggressiveness of CPLEX presolve and cut generation procedures, etc.), in order to do a fair comparison, we include an empty user-cut callback function (that does not separate any user cuts) whenever we test a formulation which does not employ cuts, e.g., basic formulation from the literature, i.e., from~\citet{chen2018data}. {We have also conducted tests under the default CPLEX settings without the empty user cut callback to confirm that our overall conclusions do not change under the default settings.}
Our preliminary tests indicated that separating a large number of inequalities throughout the branch-and-cut tree slows down the search process, so we separate our inequalities only at the root node. 

The maximum value of $\theta$ at which~\eqref{eq:dr-ccp} is feasible can be computed by solving a variant of~\eqref{eq:joint}, that is, with the same set of constraints~\eqref{joint:vars}--\eqref{joint:bigM2} but treating $\theta$ as a variable to be maximized. We first describe our test instances in Section \ref{subsec:instances} and then discuss the performance of our proposed approaches in Section~\ref{subsec:results}.

\subsection{Test instances}\label{subsec:instances}

We consider the distributionally robust chance-constrained formulation of a transportation problem from~\citet{chen2018data}. This is the problem of transshipping a single good from a set of factories $[F]$ to a set of distribution centers $[D]$ to meet their demands while minimizing the transportation cost. Each factory $f\in[F]$ has an individual production capacity $m_f$, each distribution center $d\in[D]$ faces a random demand $\xi_d$ from the end customers, and transshipping one unit of the good from factory $f$ to distribution center incurs a cost $c_{fd}$. Given $N$ samples $\left\{ \vxi_i=(\xi_{id})_{d\in[D]}: i\in[N]  \right\}$ of $\vxi$, this problem is given by 
\begin{subequations}\label{eq:transportation}
	\begin{align}
	\min\quad & \vc^\top \vx   \label{transport:obj}\\
	\text{s.t.} \quad & \bbP \left[ \sum\limits_{f\in[F]} x_{fd} \geq {\xi}_d,\quad \forall d \in [D] \right] \geq 1-\epsilon, \quad \bbP \in \mathcal{F}(\theta), \label{transport:cc-demand}\\
	& \sum\limits_{d\in[D]} x_{fd} \leq m_f, \quad f \in [F], \label{transport:capacity}\\
	& x_{fd} \geq 0, \quad  f \in [F],\ d\in [D]. \label{transport:vars}
	\end{align}
\end{subequations}
Here, \eqref{transport:cc-demand} is a joint chance constraint with right-hand side uncertainty, so~\eqref{eq:transportation} can be reformulated as in~\eqref{eq:joint}.

We use the same random instance generation scheme from~\citet{chen2018data}. We generate instances with {$F\in\{5,10\}$} factories and  
{$D\in\{50,100\}$} distribution centers whose locations are chosen uniformly at random from the Euclidean plane $[0,10]^2$. We set the transportation cost $c_{fd}$ to the Euclidean distance between factory $f\in[F]$ and distribution center $d\in[D]$. We obtain the scenarios by sampling for the demand vector $\vxi$ from a uniform distribution supported on $\left[0.8\vmu, 1.2\vmu\right]$, where the expected demand $\mu_d$ of any distribution center $d\in[D]$ is chosen uniformly at random from $[0,10]$. The capacity $m_f$ of each factory $f\in[F]$ is drawn uniformly from $[0,1]$ at first, but the capacities are scaled later so that the total capacity $\sum_{f\in[F]}m_f$ equals $\frac{3}{2}\max_{i\in[N]}\left\{\sum_{d\in[D]}\xi_{id}\right\}$.  For each instance, we test ten different values $\theta_1<\cdots<\theta_{10}$ for the Wasserstein radius. As in~\citet{chen2018data}, we set $\theta_1=0.001$.   For the other values, we compute the maximum value $\theta_{\max}$ of $\theta$ such that~\eqref{eq:joint} is feasible and set $\theta_j=\frac{j-1}{10}\theta_{\max}$ for $j=2,\ldots, 10$. 
We have empirically found that the value $\theta_{\max}$ is between 0.1 and 0.35 for our instances, so $\theta_2$ is between 0.01 and 0.035 and thus greater than $\theta_1=0.001$.
We fix the risk tolerance $\epsilon$ to be 0.1. Lastly, we select $M$ to be
\begin{equation}\label{eq:bigMvalue-transport}
M:=\max\left\{\sum_{f\in[F]}m_f-\min_{i\in[N],d\in[D]}\left\{\xi_{id}\right\},\ \max_{i\in[N],d\in[D]}\left\{\xi_{id}\right\}\right\},
\end{equation}
which is sufficiently large (see Remark~\ref{rem:bigM-joint}). For each problem parameter combination, we generate 10 random instances and report the average statistics.

As noted in \citet{chen2018data}, there is no need to specify which norm $\|\cdot\|$ to use in~\eqref{eq:joint} and~\eqref{eq:joint-knapsack} when reformulating~\eqref{eq:transportation} because the random right-hand side inside~\eqref{transport:cc-demand} contains a single random variable with coefficient 1 so that all $\|\vb_p\|_*$ in~\eqref{eq:joint} and~\eqref{eq:joint-knapsack} equal 1 in the instances generated.

\subsection{Performance Analysis} \label{subsec:results}

\newcommand{\Basic}{\texttt{Basic}\xspace}
\newcommand{\Improved}{\texttt{Improved}\xspace}
\newcommand{\Mixing}{\texttt{Mixing}\xspace}
\newcommand{\Path}{\texttt{Path}\xspace}
\newcommand{\MixingPath}{\texttt{Mixing+Path}\xspace}

In this section, we summarize our experiments with radius $\theta = \theta_1,\dots,\theta_{10}$ and  the number of samples $N=100, {1000}, 3000$. We compare the following five formulations:
\begin{description}
	\item[\Basic:] the basic formulation~\eqref{eq:joint} given by~\citet{chen2018data} with the big-$M$ value computed as in Remark \ref{rem:bigM-joint},
	\item[\Improved:] the improved formulation~\eqref{eq:joint-k-compact},
	\item[\Mixing:] the improved formulation~\eqref{eq:joint-k-compact} with the mixing inequalities~\eqref{eq:mixing-rhs},
	\item[\Path:] the improved formulation~\eqref{eq:joint-k-compact} with the path inequalities~\eqref{eq:path}, and
	\item[\MixingPath:] the improved formulation~\eqref{eq:joint-k-compact} with both mixing~\eqref{eq:mixing-rhs} and path inequalities \eqref{eq:path}.
\end{description}

In Tables \ref{tab:F5-D50-timegapcuts} and \ref{tab:F10-D100-timegapcuts}, the following statistics are reported:
\begin{description}
	\item[Time(Gap):] the average solution time (in seconds measured externally from CPLEX by C++) of the instances that were solved to optimality, and, in parentheses, the average of the final optimality gap of the instances that were not solved to optimality within the CPLEX time limit.  The optimality gap is computed as $(UB-LB)/LB*100$ where $UB$ and $LB$ respectively are the objective values of the best feasible solution and the best lower bound value at the corresponding time of the solution process. A `*' in a Time or Gap entry indicates that either no instance was solved to optimality or all instances were solved to optimality within the CPLEX time limit so that there were no instances for measuring the corresponding statistic.
	
	We also report in this column the number
	of instances solved to optimality within the CPLEX time limit, $s$, and the number of instances for which a feasible solution was found $f$. Since for most cases, we observed that $s=f=10$, we add $[s/f]$ in front of the Time(Gap) statistic only when $s<10$ or $f < 10$.
	
	A `n/a' entry for the Time(Gap) statistic denotes when no feasible solution was found in any of the 10 instances.
	
	\item[Cuts:] the average number of cuts added for \Mixing, \Path, and \MixingPath. For \MixingPath, the `Cuts' are broken down into the number of mixing inequalities and the number of path inequalities added respectively.
\end{description}

In Tables \ref{tab:F5-D50-root} and \ref{tab:F10-D100-root}, the following statistics are reported:
\begin{description}
	\item[R.time:] the average time spent (in seconds measured externally from CPLEX by C++) at the root node of the branch-and-bound tree over all instances. A `n/a' entry indicates that no feasible solution was found in any of the 10 instances within the CPLEX time limit.
	\item[R.gap:] the final optimality gap at the root node of the branch-and-bound tree. A `n/a' entry indicates that no solution was found in any of the 10 instances within the CPLEX time limit.
\end{description}

The results in all tables highlight that when the radius $\theta$ is small, the resulting problems are much harder to solve. This was also reported in~\citet{chen2018data}.

When $N=100$, {$F=5, D=50$,}  and $\theta=\theta_1$, Table~\ref{tab:F5-D50-timegapcuts} shows that \Basic does not finish within the CPLEX time limit of 3600 seconds for any of the 10 randomly generated instances and terminates with a 1.16\% optimality gap, on average. In contrast, \Improved solves all instances to optimality in under 5 seconds on average. In fact, \Improved is so effective that it does not leave much room for improvement for the additional valid inequalities in \Mixing, \Path and \MixingPath, and the separation of the valid inequalities results in a slight increase in the solution time in most cases. Nevertheless, the latter three formulations solve all instances to optimality in under 9 seconds on average. When $\theta\geq\theta_2$, \Basic solves all instances to optimality in under 27 seconds on average, but all other formulations solve in under 0.06 seconds on average. {In Appendix~\ref{sec:app}, we provide supplementary results that show that the mixing and path inequalities are very effective when added to \Basic, but \Improved without any valid inequalities performs better than \Basic with these inequalities.}

{To test the scalability of our proposed approaches with respect to the number of scenarios, in Table~\ref{tab:F5-D50-timegapcuts}}, we {also} report the performance of the five formulations on instances with {$N=1000$ and } $N=3000$, respectively for  {$F=5, D=50$}. {In the out-of-sample tests for these instances reported in \cite{chen2018data}, for $N=1000$, the authors state the difficulty of solving the problem to proven optimality and report their results with a not  necessarily optimal solution obtained after a couple of minutes of computing. Our results in Table~\ref{tab:F5-D50-timegapcuts} show that   our proposed formulation  provides provably optimal solutions  within less than a minute in all instances but those with $\theta_1$.   
Furthermore, as indicated in Table~\ref{tab:F5-D50-timegapcuts}, we can scale up the number of scenarios even further.}  In the experiments with  $N=3000$, we observe that even for the largest value of $\theta_{10}$, 
\Basic was unable to solve any of the ten instances within the CPLEX time limit, while all of the new formulations \Improved, \Mixing, \Path and \MixingPath solved all instances to optimality for $\theta \geq \theta_3$ with an average time of at most 20 seconds. For $\theta \leq \theta_2$, our new formulations did not manage to solve any instances to optimality, but did reduce the average optimality gap to at most 0.8\%. In contrast, \Basic failed to solve all instances for $\theta \leq \theta_2$, and in fact did not even find a feasible solution within the time limit for many instances. For the instances where a feasible solution was found, the gap remained quite large at $\approx 70\%$ at termination.

{We also test the scalability of various methods with respect to the number of original decision variables in the original problem. To this end, we let $F=10, D=100$  and report our results in Table~\ref{tab:F10-D100-timegapcuts} for $N=100,1000,3000$. (Note that an increase in $D$ implies an increase in the dimension of the random data, which may in turn require an increase in sample size to ensure the same out-of-sample performance.) Comparing Tables~\ref{tab:F5-D50-timegapcuts} and \ref{tab:F10-D100-timegapcuts}, we see that not surprisingly, the problems with a larger number of original decision variables are harder to solve.  Fewer instances can be solved to optimality with \Basic. In fact, \Basic cannot even obtain a feasible solution after an hour of computing for instances with $F=10, D=100, N=3000.$  In contrast,  \Improved solves  all instances  to optimality in less than a minute except for the ones with $\theta_1$ when $N=100,1000$ and with $\theta\leq\theta_5$ in the case of $N=3000$. Indeed, we observe a very slight increase in solution times for our proposed formulations in the case of $F=10, D=100$ in contrast to $F=5, D=50$, and all trends reported for $F=5, D=50$ remain the same in the case of $F=10, D=100.$ }

Tables  \ref{tab:F5-D50-timegapcuts} and \ref{tab:F10-D100-timegapcuts} also give us insight into the marginal effect of each class of inequalities. We observe that mixing inequalities are only generated for $\theta=\theta_1$ (plus a total of 3 inequalities across all instances when $\theta=\theta_2$ and $N=100$). When the radius of the Wasserstein ball is small, the nominal region $\cX_{\SAA}(\cS)$ is a better approximation for the distributionally robust region $\cX_{\DR}(\cS)$. Since mixing inequalities are valid for $\cX_{\SAA}(\cS)$, it is expected that they have stronger effects as for smaller radius $\theta$, thus the observed behavior is not surprising. As mentioned above, when $N=100$, the \Improved formulation is already so effective that separating inequalities slightly increases solution times. However, when $N=3000$, $F=5$ and $D=50$, from Table \ref{tab:F5-D50-timegapcuts}, we see that while \Improved is still quite effective, and manages to reduce the gap to $0.78\%$ for $\theta=\theta_1$, separating the valid inequalities reduces this further ($0.52\%$ for \Mixing, $0.62\%$ for \Path, and $0.48\%$ for \MixingPath). We see a similar phenomenon for $\theta = \theta_2$, but only path inequalities are separated. {Similar observations can be made for $N=3000$, $F=10$, $D=100$.}

Tables \ref{tab:F5-D50-root} {and  \ref{tab:F10-D100-root}} provide further information on the performance of the formulations at the root node  of the branch-and-bound tree.  
From Table  \ref{tab:F5-D50-root}, we see that the root gap of the new formulations \Improved, \Mixing, \Path and \MixingPath are at most $0.8\%$ on average for the most difficult regime $\theta=\theta_1,\theta_2$ and $N=3000$. This is significantly better than the root gap of \Basic.
In fact, in our experiments, we have observed that for these types of distributionally robust CCPs, the  branch-and-bound process is very ineffective in terms of reducing the remaining gap, with only a small difference between root gap and final gap. This also highlights the importance of starting off with very strong formulations. In fact, we observed that for $\theta \geq \theta_2$, the gap threshold of $0.01\%$ is achieved at the root node for our improved formulations.

\section{Conclusion}\label{sec:conclusion}

This paper studies in detail the formulation for distributionally robust chance-constrained programs under Wasserstein ambiguity, focusing on the case of linear safety sets with right-hand side uncertainty. We reveal a hidden connection with nominal chance-constraints, and provide a class of valid inequalities which are exactly the mixing inequalities for the nominal chance constraint. We also adapt the quantile strengthening technique to the distributionally robust setting, making one set of constraints \eqref{joint:bigM2} in the original MIP formulation \eqref{eq:joint} big-$M$-free. We then exploit the CVaR interpretation of the distributionally robust constraint \eqref{eq:cc-distance-cvar} to provide an improved formulation \eqref{eq:joint-k-compact} with significantly fewer constraints. Finally, we uncover a mixed-integer substructure which has been studied in the context of robust 0-1 programming \citet{atamturk2006robust}, and provide second class of valid inequalities based on this connection.

Our computational results demonstrate the benefit of our improved formulation and valid inequalities. Solution times are drastically reduced and larger problems  can now be solved in seconds, e.g.,  problems with thousands of scenarios (rather than hundreds). 

\textbf{Acknowledgments.}
This paper is in memory of Shabbir Ahmed, whose fundamental contributions on mixing sets, chance-constrained programming and distributionally robust optimization we build upon. We thank the two referees and the AE for their suggestions that improved the exposition. 
This research is supported, in part, by ONR grant N00014-19-1-2321, by the Institute for Basic Science (IBS-R029-C1), Award N660011824020 from the DARPA Lagrange Program and NSF Award 1740707.

\begin{landscape}

\begin{table}
	\centering
	\caption{Results for $F=5$, $D=50$, $\epsilon=0.1$}
	\label{tab:F5-D50-timegapcuts}
	\begin{tabular}{ll|r|r|rr|rr|rr}
		\toprule
		&               & \multicolumn{1}{c}{\texttt{Basic}} &  \multicolumn{1}{c}{\texttt{Improved}} & \multicolumn{2}{c}{\texttt{Mixing}} & \multicolumn{2}{c}{\texttt{Path}} & \multicolumn{2}{c}{\texttt{Mixing+Path}} \\
		$N$ & $\theta$ &                 Time(Gap) &                Time(Gap) &                Time(Gap) &    Cuts &                Time(Gap) &    Cuts &                Time(Gap) &           Cuts \\
		\midrule
		\multirow{10}{*}{100} & $\theta_{1}$ &   [0/10] \ \hfill *(1.16) &                  4.29(*) &                  6.64(*) &    54.7 &                  7.31(*) &   287.6 &                  8.40(*) &     41.7/274.6 \\
		& $\theta_{2}$ &                  26.58(*) &                  0.04(*) &                  0.05(*) &     0.3 &                  0.06(*) &    88.2 &                  0.06(*) &       0.3/88.2 \\
		& $\theta_{3}$ &                   4.27(*) &                  0.04(*) &                  0.04(*) &     0.0 &                  0.04(*) &    73.8 &                  0.05(*) &       0.0/73.8 \\
		& $\theta_{4}$ &                   3.36(*) &                  0.03(*) &                  0.03(*) &     0.0 &                  0.03(*) &    44.4 &                  0.03(*) &       0.0/44.4 \\
		& $\theta_{5}$ &                   1.34(*) &                  0.03(*) &                  0.03(*) &     0.0 &                  0.03(*) &    42.4 &                  0.03(*) &       0.0/42.4 \\
		& $\theta_{6}$ &                   1.17(*) &                  0.03(*) &                  0.03(*) &     0.0 &                  0.03(*) &    43.1 &                  0.03(*) &       0.0/43.1 \\
		& $\theta_{7}$ &                   1.04(*) &                  0.03(*) &                  0.03(*) &     0.0 &                  0.03(*) &    40.3 &                  0.03(*) &       0.0/40.3 \\
		& $\theta_{8}$ &                   0.94(*) &                  0.03(*) &                  0.03(*) &     0.0 &                  0.03(*) &    40.2 &                  0.03(*) &       0.0/40.2 \\
		& $\theta_{9}$ &                   0.86(*) &                  0.03(*) &                  0.03(*) &     0.0 &                  0.03(*) &    36.4 &                  0.03(*) &       0.0/36.4 \\
		& $\theta_{10}$ &                   0.84(*) &                  0.02(*) &                  0.03(*) &     0.0 &                  0.03(*) &    36.0 &                  0.03(*) &       0.0/36.0 \\
		\cline{1-10}
		\multirow{10}{*}{1000} & $\theta_{1}$ &   [0/6] \ \hfill *(97.90) &  [0/10] \ \hfill *(0.59) &  [0/10] \ \hfill *(0.52) &   532.8 &  [0/10] \ \hfill *(0.44) &  2472.5 &  [0/10] \ \hfill *(0.44) &   381.0/2377.9 \\
		& $\theta_{2}$ &  [0/10] \ \hfill *(66.04) &                 36.84(*) &                 36.68(*) &     0.0 &                 40.41(*) &   114.3 &                 40.61(*) &      0.0/114.3 \\
		& $\theta_{3}$ &  [0/10] \ \hfill *(31.10) &                  9.15(*) &                  9.33(*) &     0.0 &                  9.23(*) &    82.9 &                  9.61(*) &       0.0/82.9 \\
		& $\theta_{4}$ &   [0/10] \ \hfill *(2.11) &                  7.73(*) &                  7.60(*) &     0.0 &                  7.74(*) &    61.6 &                  7.71(*) &       0.0/61.6 \\
		& $\theta_{5}$ &                 287.92(*) &                  6.12(*) &                  6.16(*) &     0.0 &                  6.20(*) &    58.5 &                  6.19(*) &       0.0/58.5 \\
		& $\theta_{6}$ &                 130.38(*) &                  5.83(*) &                  5.83(*) &     0.0 &                  5.95(*) &    49.0 &                  5.94(*) &       0.0/49.0 \\
		& $\theta_{7}$ &                 117.77(*) &                  5.75(*) &                  5.72(*) &     0.0 &                  5.77(*) &    48.9 &                  5.75(*) &       0.0/48.9 \\
		& $\theta_{8}$ &                  99.95(*) &                  5.58(*) &                  5.52(*) &     0.0 &                  5.65(*) &    49.0 &                  5.64(*) &       0.0/49.0 \\
		& $\theta_{9}$ &                  90.82(*) &                  5.27(*) &                  5.27(*) &     0.0 &                  5.34(*) &    49.0 &                  5.39(*) &       0.0/49.0 \\
		& $\theta_{10}$ &                  78.88(*) &                  4.95(*) &                  4.97(*) &     0.0 &                  5.04(*) &    49.1 &                  5.02(*) &       0.0/49.1 \\
		\cline{1-10}
		\multirow{10}{*}{3000} & $\theta_{1}$ &          [0/0] \hfill n/a &  [0/10] \ \hfill *(0.78) &  [0/10] \ \hfill *(0.52) &  2358.2 &  [0/10] \ \hfill *(0.62) &  6536.0 &  [0/10] \ \hfill *(0.48) &  1470.3/4228.1 \\
		& $\theta_{2}$ &   [0/5] \ \hfill *(69.56) &  [0/10] \ \hfill *(0.49) &  [0/10] \ \hfill *(0.49) &     0.0 &  [0/10] \ \hfill *(0.41) &  6324.6 &  [0/10] \ \hfill *(0.41) &     0.0/6102.2 \\
		& $\theta_{3}$ &   [0/4] \ \hfill *(48.65) &                 17.89(*) &                 17.62(*) &     0.0 &                 18.94(*) &   200.8 &                 18.29(*) &      0.0/200.8 \\
		& $\theta_{4}$ &   [0/4] \ \hfill *(15.01) &                 13.74(*) &                 13.22(*) &     0.0 &                 14.03(*) &    94.1 &                 13.94(*) &       0.0/94.1 \\
		& $\theta_{5}$ &   [0/10] \ \hfill *(1.11) &                 12.75(*) &                 12.60(*) &     0.0 &                 13.65(*) &    88.3 &                 13.55(*) &       0.0/88.3 \\
		& $\theta_{6}$ &   [0/10] \ \hfill *(0.58) &                 12.29(*) &                 12.21(*) &     0.0 &                 12.83(*) &    80.2 &                 12.68(*) &       0.0/80.2 \\
		& $\theta_{7}$ &   [0/10] \ \hfill *(0.48) &                 12.28(*) &                 12.01(*) &     0.0 &                 12.93(*) &    56.7 &                 12.95(*) &       0.0/56.7 \\
		& $\theta_{8}$ &   [0/10] \ \hfill *(0.40) &                 12.48(*) &                 12.20(*) &     0.0 &                 12.78(*) &    52.9 &                 12.65(*) &       0.0/52.9 \\
		& $\theta_{9}$ &   [0/10] \ \hfill *(0.37) &                 11.96(*) &                 11.81(*) &     0.0 &                 12.36(*) &    49.4 &                 12.52(*) &       0.0/49.4 \\
		& $\theta_{10}$ &   [0/10] \ \hfill *(0.28) &                  8.04(*) &                  8.19(*) &     0.0 &                  8.86(*) &    49.4 &                  8.94(*) &       0.0/49.4 \\
		\bottomrule
	\end{tabular}
\end{table}

\begin{table}
	\centering
	\caption{Results at root node for $F=5$, $D=50$, $\epsilon=0.1$}
	\label{tab:F5-D50-root}
	\begin{tabular}{ll|rr|rr|rr|rr|rr}
		\toprule
		&               & \multicolumn{2}{c}{\texttt{Basic}} & \multicolumn{2}{c}{\texttt{Improved}} & \multicolumn{2}{c}{\texttt{Mixing}} & \multicolumn{2}{c}{\texttt{Path}} & \multicolumn{2}{c}{\texttt{Mixing+Path}} \\
		$N$ & $\theta$ &         R.time &  R.gap &            R.time & R.gap &          R.time & R.gap &        R.time & R.gap &               R.time & R.gap \\
		\midrule
		\multirow{10}{*}{100} & $\theta_{1}$ &           5.67 &  90.53 &              0.19 &  0.34 &            0.21 &  0.35 &          0.50 &  0.31 &                 0.47 &  0.31 \\
		& $\theta_{2}$ &          11.22 &  65.19 &              0.04 &  0.00 &            0.05 &  0.00 &          0.06 &  0.00 &                 0.06 &  0.00 \\
		& $\theta_{3}$ &           2.98 &  38.94 &              0.04 &  0.00 &            0.04 &  0.00 &          0.04 &  0.00 &                 0.05 &  0.00 \\
		& $\theta_{4}$ &           2.66 &  13.35 &              0.03 &  0.00 &            0.03 &  0.00 &          0.03 &  0.00 &                 0.03 &  0.00 \\
		& $\theta_{5}$ &           1.32 &   0.04 &              0.03 &  0.00 &            0.03 &  0.00 &          0.03 &  0.00 &                 0.03 &  0.00 \\
		& $\theta_{6}$ &           1.17 &   0.00 &              0.03 &  0.00 &            0.03 &  0.00 &          0.03 &  0.00 &                 0.03 &  0.00 \\
		& $\theta_{7}$ &           1.04 &   0.00 &              0.03 &  0.00 &            0.03 &  0.00 &          0.03 &  0.00 &                 0.03 &  0.00 \\
		& $\theta_{8}$ &           0.94 &   0.00 &              0.03 &  0.00 &            0.03 &  0.00 &          0.03 &  0.00 &                 0.03 &  0.00 \\
		& $\theta_{9}$ &           0.86 &   0.00 &              0.03 &  0.00 &            0.03 &  0.00 &          0.03 &  0.00 &                 0.03 &  0.00 \\
		& $\theta_{10}$ &           0.84 &   0.00 &              0.02 &  0.00 &            0.03 &  0.00 &          0.03 &  0.00 &                 0.03 &  0.00 \\
		\cline{1-12}
		\multirow{10}{*}{1000} & $\theta_{1}$ &        2334.91 &  99.45 &            338.01 &  0.63 &          614.52 &  0.54 &       1968.43 &  0.45 &              2000.91 &  0.45 \\
		& $\theta_{2}$ &         275.76 &  73.64 &             37.57 &  0.01 &           36.76 &  0.01 &         40.65 &  0.00 &                40.48 &  0.00 \\
		& $\theta_{3}$ &         223.92 &  47.13 &              9.22 &  0.00 &            9.30 &  0.00 &          9.29 &  0.00 &                 9.59 &  0.00 \\
		& $\theta_{4}$ &         289.21 &  16.30 &              7.70 &  0.00 &            7.71 &  0.00 &          7.77 &  0.00 &                 7.77 &  0.00 \\
		& $\theta_{5}$ &         103.41 &   0.51 &              6.10 &  0.00 &            6.09 &  0.00 &          6.19 &  0.00 &                 6.18 &  0.00 \\
		& $\theta_{6}$ &          89.56 &   0.39 &              5.81 &  0.01 &            5.83 &  0.01 &          5.92 &  0.01 &                 5.86 &  0.01 \\
		& $\theta_{7}$ &          78.75 &   0.41 &              5.76 &  0.00 &            5.71 &  0.00 &          5.74 &  0.00 &                 5.80 &  0.00 \\
		& $\theta_{8}$ &          67.40 &   0.40 &              5.53 &  0.00 &            5.55 &  0.00 &          5.62 &  0.00 &                 5.65 &  0.00 \\
		& $\theta_{9}$ &          63.73 &   0.38 &              5.28 &  0.00 &            5.26 &  0.00 &          5.35 &  0.00 &                 5.36 &  0.00 \\
		& $\theta_{10}$ &          58.08 &   0.36 &              4.93 &  0.00 &            4.91 &  0.00 &          5.04 &  0.00 &                 5.03 &  0.00 \\
		\cline{1-12}
		\multirow{10}{*}{3000} & $\theta_{1}$ &            n/a &    n/a &             72.08 &  0.80 &         1931.99 &  0.52 &       3598.81 &  0.62 &              3601.05 &  0.48 \\
		& $\theta_{2}$ &        3144.09 &  70.41 &            134.46 &  0.55 &          109.72 &  0.55 &       3600.10 &  0.41 &              3600.22 &  0.41 \\
		& $\theta_{3}$ &        2952.26 &  51.31 &             17.89 &  0.01 &           17.62 &  0.01 &         18.94 &  0.01 &                18.29 &  0.01 \\
		& $\theta_{4}$ &        2684.77 &  15.72 &             13.74 &  0.01 &           13.22 &  0.01 &         14.03 &  0.01 &                13.94 &  0.01 \\
		& $\theta_{5}$ &        3181.43 &   1.14 &             12.75 &  0.00 &           12.60 &  0.00 &         13.65 &  0.00 &                13.55 &  0.00 \\
		& $\theta_{6}$ &        3176.11 &   0.63 &             12.29 &  0.00 &           12.21 &  0.00 &         12.83 &  0.00 &                12.68 &  0.00 \\
		& $\theta_{7}$ &        2958.81 &   0.55 &             12.28 &  0.01 &           12.01 &  0.01 &         12.93 &  0.01 &                12.95 &  0.01 \\
		& $\theta_{8}$ &        2876.49 &   0.47 &             12.48 &  0.01 &           12.20 &  0.01 &         12.78 &  0.01 &                12.65 &  0.01 \\
		& $\theta_{9}$ &        2781.77 &   0.45 &             11.96 &  0.01 &           11.81 &  0.01 &         12.36 &  0.01 &                12.52 &  0.01 \\
		& $\theta_{10}$ &        2439.69 &   0.41 &              8.04 &  0.01 &            8.19 &  0.01 &          8.86 &  0.01 &                 8.94 &  0.01 \\
		\bottomrule
	\end{tabular}
\end{table}

\end{landscape}
\begin{landscape}

\begin{table}
	\centering
	\caption{Results for $F=10$, $D=100$, $\epsilon=0.1$}
	\label{tab:F10-D100-timegapcuts}
	\begin{tabular}{ll|r|r|rr|rr|rr}
		\toprule
		&               & \multicolumn{1}{c}{\texttt{Basic}} &  \multicolumn{1}{c}{\texttt{Improved}} & \multicolumn{2}{c}{\texttt{Mixing}} & \multicolumn{2}{c}{\texttt{Path}} & \multicolumn{2}{c}{\texttt{Mixing+Path}} \\
		$N$ & $\theta$ &                 Time(Gap) &                Time(Gap) &                Time(Gap) &    Cuts &                Time(Gap) &    Cuts &                Time(Gap) &           Cuts \\
		\midrule
		\multirow{10}{*}{100} & $\theta_{1}$ &        [0/10] \ \hfill *(3.01) &                 10.46(*) &                 19.19(*) &    55.7 &                 29.36(*) &   228.2 &                 35.45(*) &    48.0/236.2 \\
		& $\theta_{2}$ &                       37.49(*) &                  0.11(*) &                  0.11(*) &     0.0 &                  0.12(*) &   171.3 &                  0.13(*) &     0.0/171.3 \\
		& $\theta_{3}$ &                       18.45(*) &                  0.10(*) &                  0.10(*) &     0.0 &                  0.11(*) &   146.7 &                  0.12(*) &     0.0/146.7 \\
		& $\theta_{4}$ &                       16.91(*) &                  0.09(*) &                  0.09(*) &     0.0 &                  0.10(*) &   123.2 &                  0.10(*) &     0.0/123.2 \\
		& $\theta_{5}$ &                        9.17(*) &                  0.09(*) &                  0.09(*) &     0.0 &                  0.09(*) &   119.0 &                  0.09(*) &     0.0/119.0 \\
		& $\theta_{6}$ &                        5.71(*) &                  0.06(*) &                  0.06(*) &     0.0 &                  0.07(*) &    84.7 &                  0.07(*) &      0.0/84.7 \\
		& $\theta_{7}$ &                        4.74(*) &                  0.06(*) &                  0.06(*) &     0.0 &                  0.07(*) &    86.3 &                  0.07(*) &      0.0/86.3 \\
		& $\theta_{8}$ &                        3.78(*) &                  0.06(*) &                  0.06(*) &     0.0 &                  0.07(*) &    74.5 &                  0.07(*) &      0.0/74.5 \\
		& $\theta_{9}$ &                        3.20(*) &                  0.06(*) &                  0.06(*) &     0.0 &                  0.06(*) &    69.9 &                  0.06(*) &      0.0/69.9 \\
		& $\theta_{10}$ &                        2.89(*) &                  0.05(*) &                  0.06(*) &     0.0 &                  0.06(*) &    54.8 &                  0.06(*) &      0.0/54.8 \\
		\cline{1-10}
		\multirow{10}{*}{1000} & $\theta_{1}$ &        [0/1] \ \hfill *(98.73) &  [0/10] \ \hfill *(0.45) &  [0/10] \ \hfill *(0.42) &   382.3 &  [0/10] \ \hfill *(0.39) &  2150.7 &  [0/10] \ \hfill *(0.39) &  249.6/2032.4 \\
		& $\theta_{2}$ &        [0/7] \ \hfill *(67.05) &                 27.02(*) &                 26.05(*) &     0.0 &                 37.06(*) &   399.8 &                 36.97(*) &     0.0/399.8 \\
		& $\theta_{3}$ &        [0/4] \ \hfill *(34.22) &                 14.44(*) &                 14.73(*) &     0.0 &                 15.55(*) &   231.1 &                 14.70(*) &     0.0/231.1 \\
		& $\theta_{4}$ &        [0/5] \ \hfill *(10.48) &                 10.18(*) &                 10.41(*) &     0.0 &                 12.13(*) &   186.3 &                 12.07(*) &     0.0/186.3 \\
		& $\theta_{5}$ &  [4/10] \ \hfill 2599.60(0.45) &                 10.37(*) &                 10.68(*) &     0.0 &                 11.66(*) &   186.1 &                 11.31(*) &     0.0/186.1 \\
		& $\theta_{6}$ &                      968.46(*) &                  3.14(*) &                  3.25(*) &     0.0 &                  3.29(*) &    96.3 &                  3.28(*) &      0.0/96.3 \\
		& $\theta_{7}$ &                      853.68(*) &                  2.91(*) &                  3.06(*) &     0.0 &                  3.48(*) &    96.4 &                  3.52(*) &      0.0/96.4 \\
		& $\theta_{8}$ &                      647.80(*) &                  3.08(*) &                  3.11(*) &     0.0 &                  3.12(*) &    96.3 &                  3.07(*) &      0.0/96.3 \\
		& $\theta_{9}$ &                      528.73(*) &                  2.78(*) &                  2.74(*) &     0.0 &                  2.81(*) &    96.2 &                  2.85(*) &      0.0/96.2 \\
		& $\theta_{10}$ &                      375.32(*) &                  2.58(*) &                  2.57(*) &     0.0 &                  2.60(*) &    96.5 &                  2.59(*) &      0.0/96.5 \\
		\cline{1-10}
		\multirow{10}{*}{3000} & $\theta_{1}$ &               [0/0] \hfill n/a &  [0/10] \ \hfill *(0.59) &  [0/10] \ \hfill *(0.48) &  1494.1 &  [0/10] \ \hfill *(0.46) &  6003.2 &  [0/10] \ \hfill *(0.44) &  948.3/5356.1 \\
		& $\theta_{2}$ &               [0/0] \hfill n/a &  [0/10] \ \hfill *(0.49) &  [0/10] \ \hfill *(0.49) &     0.0 &  [0/10] \ \hfill *(0.42) &  6117.0 &  [0/10] \ \hfill *(0.42) &    0.0/6029.7 \\
		& $\theta_{3}$ &               [0/0] \hfill n/a &                407.31(*) &                336.00(*) &     0.0 &                541.10(*) &  1422.3 &                642.69(*) &    0.0/1422.3 \\
		& $\theta_{4}$ &               [0/0] \hfill n/a &                158.82(*) &                158.30(*) &     0.0 &                148.37(*) &   386.3 &                159.23(*) &     0.0/386.3 \\
		& $\theta_{5}$ &               [0/0] \hfill n/a &                 64.38(*) &                 64.42(*) &     0.0 &                 99.60(*) &   293.8 &                100.18(*) &     0.0/293.8 \\
		& $\theta_{6}$ &               [0/0] \hfill n/a &                 33.52(*) &                 33.99(*) &     0.0 &                 35.09(*) &    98.4 &                 35.16(*) &      0.0/98.4 \\
		& $\theta_{7}$ &               [0/0] \hfill n/a &                 33.33(*) &                 33.73(*) &     0.0 &                 35.30(*) &    98.6 &                 35.30(*) &      0.0/98.6 \\
		& $\theta_{8}$ &               [0/0] \hfill n/a &                 34.34(*) &                 34.81(*) &     0.0 &                 36.85(*) &    98.7 &                 37.01(*) &      0.0/98.7 \\
		& $\theta_{9}$ &               [0/0] \hfill n/a &                 35.38(*) &                 35.31(*) &     0.0 &                 36.69(*) &    98.6 &                 36.63(*) &      0.0/98.6 \\
		& $\theta_{10}$ &               [0/0] \hfill n/a &                 26.28(*) &                 25.79(*) &     0.0 &                 26.19(*) &    98.8 &                 25.76(*) &      0.0/98.8 \\
		\bottomrule
	\end{tabular}
\end{table}

\begin{table}
	\centering
	\caption{Results at root node for $F=10$, $D=100$, $\epsilon=0.1$}
	\label{tab:F10-D100-root}
	\begin{tabular}{ll|rr|rr|rr|rr|rr}
		\toprule
		&               & \multicolumn{2}{c}{\texttt{Basic}} & \multicolumn{2}{c}{\texttt{Improved}} & \multicolumn{2}{c}{\texttt{Mixing}} & \multicolumn{2}{c}{\texttt{Path}} & \multicolumn{2}{c}{\texttt{Mixing+Path}} \\
		$N$ & $\theta$ &         R.time &  R.gap &            R.time & R.gap &          R.time & R.gap &        R.time & R.gap &               R.time & R.gap \\
		\midrule
		\multirow{10}{*}{100} & $\theta_{1}$ &          29.20 &  90.99 &              0.38 &  0.42 &            0.35 &  0.44 &          0.46 &  0.43 &                 0.50 &  0.43 \\
		& $\theta_{2}$ &          10.39 &  68.59 &              0.11 &  0.00 &            0.11 &  0.00 &          0.12 &  0.00 &                 0.13 &  0.00 \\
		& $\theta_{3}$ &          14.11 &  43.85 &              0.10 &  0.00 &            0.10 &  0.00 &          0.11 &  0.00 &                 0.12 &  0.00 \\
		& $\theta_{4}$ &          14.50 &  21.52 &              0.09 &  0.00 &            0.09 &  0.00 &          0.10 &  0.00 &                 0.10 &  0.00 \\
		& $\theta_{5}$ &           8.97 &   1.12 &              0.09 &  0.00 &            0.09 &  0.00 &          0.09 &  0.00 &                 0.09 &  0.00 \\
		& $\theta_{6}$ &           5.66 &   0.01 &              0.06 &  0.00 &            0.06 &  0.00 &          0.07 &  0.00 &                 0.07 &  0.00 \\
		& $\theta_{7}$ &           4.76 &   0.01 &              0.06 &  0.00 &            0.06 &  0.00 &          0.07 &  0.00 &                 0.07 &  0.00 \\
		& $\theta_{8}$ &           3.78 &   0.00 &              0.06 &  0.00 &            0.06 &  0.00 &          0.07 &  0.00 &                 0.07 &  0.00 \\
		& $\theta_{9}$ &           3.20 &   0.00 &              0.06 &  0.00 &            0.06 &  0.00 &          0.06 &  0.00 &                 0.06 &  0.00 \\
		& $\theta_{10}$ &           2.89 &   0.01 &              0.05 &  0.00 &            0.06 &  0.00 &          0.06 &  0.00 &                 0.06 &  0.00 \\
		\cline{1-12}
		\multirow{10}{*}{1000} & $\theta_{1}$ &        3542.86 &  99.48 &             29.81 &  0.50 &           35.95 &  0.45 &        148.13 &  0.41 &               157.22 &  0.41 \\
		& $\theta_{2}$ &        1185.36 &  72.76 &             27.02 &  0.01 &           26.05 &  0.01 &         37.06 &  0.00 &                36.97 &  0.00 \\
		& $\theta_{3}$ &         990.87 &  46.39 &             14.44 &  0.01 &           14.73 &  0.01 &         15.55 &  0.01 &                14.70 &  0.01 \\
		& $\theta_{4}$ &        1547.07 &  21.38 &             10.18 &  0.01 &           10.41 &  0.01 &         12.13 &  0.00 &                12.07 &  0.00 \\
		& $\theta_{5}$ &         570.64 &   0.84 &             10.37 &  0.00 &           10.68 &  0.00 &         11.66 &  0.00 &                11.31 &  0.00 \\
		& $\theta_{6}$ &         498.69 &   0.41 &              3.14 &  0.00 &            3.25 &  0.00 &          3.29 &  0.00 &                 3.28 &  0.00 \\
		& $\theta_{7}$ &         453.85 &   0.43 &              2.91 &  0.00 &            3.06 &  0.00 &          3.48 &  0.00 &                 3.52 &  0.00 \\
		& $\theta_{8}$ &         412.37 &   0.40 &              3.08 &  0.00 &            3.11 &  0.00 &          3.12 &  0.00 &                 3.07 &  0.00 \\
		& $\theta_{9}$ &         366.56 &   0.36 &              2.78 &  0.00 &            2.74 &  0.00 &          2.81 &  0.00 &                 2.85 &  0.00 \\
		& $\theta_{10}$ &         314.52 &   0.25 &              2.58 &  0.00 &            2.57 &  0.00 &          2.60 &  0.00 &                 2.59 &  0.00 \\
		\cline{1-12}
		\multirow{10}{*}{3000} & $\theta_{1}$ &            n/a &    n/a &            289.73 &  0.60 &          902.49 &  0.48 &       3599.15 &  0.46 &              3603.49 &  0.44 \\
		& $\theta_{2}$ &            n/a &    n/a &            456.10 &  0.52 &          368.56 &  0.52 &       3601.21 &  0.42 &              3597.16 &  0.42 \\
		& $\theta_{3}$ &            n/a &    n/a &            407.31 &  0.01 &          336.00 &  0.01 &        541.10 &  0.01 &               642.69 &  0.01 \\
		& $\theta_{4}$ &            n/a &    n/a &            158.82 &  0.01 &          158.30 &  0.01 &        148.37 &  0.01 &               159.23 &  0.01 \\
		& $\theta_{5}$ &            n/a &    n/a &             64.38 &  0.01 &           64.42 &  0.01 &         99.60 &  0.01 &               100.18 &  0.01 \\
		& $\theta_{6}$ &            n/a &    n/a &             33.52 &  0.01 &           33.99 &  0.01 &         35.09 &  0.01 &                35.16 &  0.01 \\
		& $\theta_{7}$ &            n/a &    n/a &             33.33 &  0.00 &           33.73 &  0.00 &         35.30 &  0.00 &                35.30 &  0.00 \\
		& $\theta_{8}$ &            n/a &    n/a &             34.34 &  0.00 &           34.81 &  0.00 &         36.85 &  0.00 &                37.01 &  0.00 \\
		& $\theta_{9}$ &            n/a &    n/a &             35.38 &  0.00 &           35.31 &  0.00 &         36.69 &  0.00 &                36.63 &  0.00 \\
		& $\theta_{10}$ &            n/a &    n/a &             26.28 &  0.00 &           25.79 &  0.00 &         26.19 &  0.00 &                25.76 &  0.00 \\
		\bottomrule
	\end{tabular}
\end{table}

\end{landscape}

\appendix
\newpage

\section{Supplementary Numerical Results} \label{sec:app}

In Table \ref{tab:supp-F5-D50} we report the performance of the basic formulation~\eqref{eq:joint} when combined with mixing~\eqref{eq:mixing-rhs} and path inequalities \eqref{eq:path} for $F=5$, $D=50$, $\epsilon=0.1$. These  results highlight that mixing and path inequalities are indeed useful when applied to the basic formulation, without all the other enhancements we propose.
	
We see that for $N=100$, \texttt{Basic+Mixing+Path} solves all instances (and with slightly quicker times on average) whereas \texttt{Basic} in Table \ref{tab:F5-D50-timegapcuts} does not manage to solve any instances for $\theta_1$. For $N=1000$, we again see an improvement in the number of instances solved when using the inequalities, but for larger $\theta_8,\theta_9,\theta_{10}$, times are slightly slower. We believe this is due to the extra time required to solve the larger LP relaxations as a result of adding cuts. For $N=3000$, we now see that \texttt{Basic+Mixing+Path} is unable to find a feasible integer solution within one hour for a larger number of instances than \texttt{Basic} (we again believe this is due to larger LP relaxations), but whenever it does, it often solves to optimality, which \texttt{Basic} never does. 

In contrast to \texttt{Improved}, there are more cuts generated for \texttt{Basic} for all $\theta$ values. Nevertheless, as expected, the performance achieved by \texttt{Basic+Mixing+Path} is still worse than the performance of our improved formulation \eqref{eq:joint-k-compact}, with or without inequalities.

\begin{table}[h!]
	\centering
	\caption{Supplementary results for $F=5$, $D=50$, $\epsilon=0.1$}
	\label{tab:supp-F5-D50}
	\begin{tabular}{ll|rrrr}
		\toprule
		&               & \multicolumn{4}{c}{\texttt{Basic+Mixing+Path}} \\
		$N$ & $\theta$ & Time(Gap) &   R.time & R.gap & Cuts \\
		\midrule
		\multirow{10}{*}{100} & $\theta_{1}$ &                      36.64(*) &     2.77 &  0.30 &    116.0/386.2 \\
		& $\theta_{2}$ &                       0.61(*) &     0.61 &  0.00 &     65.3/148.7 \\
		& $\theta_{3}$ &                       0.64(*) &     0.64 &  0.00 &     59.5/128.5 \\
		& $\theta_{4}$ &                       0.53(*) &     0.53 &  0.00 &     57.7/118.2 \\
		& $\theta_{5}$ &                       0.56(*) &     0.56 &  0.00 &     57.1/107.9 \\
		& $\theta_{6}$ &                       0.53(*) &     0.53 &  0.00 &     55.5/101.0 \\
		& $\theta_{7}$ &                       0.43(*) &     0.43 &  0.00 &      54.0/97.8 \\
		& $\theta_{8}$ &                       0.46(*) &     0.46 &  0.00 &      53.2/96.9 \\
		& $\theta_{9}$ &                       0.51(*) &     0.51 &  0.00 &      52.7/95.6 \\
		& $\theta_{10}$ &                       0.48(*) &     0.48 &  0.00 &      51.7/92.4 \\
		\cline{1-6}
		\multirow{10}{*}{1000} & $\theta_{1}$ &        [0/9] \ \hfill *(0.45) &  2532.66 &  0.45 &  1047.7/3462.8 \\
		& $\theta_{2}$ &                     272.45(*) &   272.45 &  0.00 &   320.0/1081.3 \\
		& $\theta_{3}$ &                     142.32(*) &   142.32 &  0.00 &    220.2/620.2 \\
		& $\theta_{4}$ &                     125.95(*) &   125.95 &  0.00 &    196.1/570.2 \\
		& $\theta_{5}$ &                     111.84(*) &   111.84 &  0.00 &    186.4/521.6 \\
		& $\theta_{6}$ &                     106.86(*) &   106.86 &  0.00 &    172.4/499.5 \\
		& $\theta_{7}$ &                     100.38(*) &   100.38 &  0.00 &    162.2/461.5 \\
		& $\theta_{8}$ &                     103.55(*) &   103.55 &  0.00 &    155.5/496.6 \\
		& $\theta_{9}$ &                     109.29(*) &   109.29 &  0.00 &    144.7/464.5 \\
		& $\theta_{10}$ &                     107.14(*) &   107.14 &  0.00 &    136.5/462.2 \\
		\cline{1-6}
		\multirow{10}{*}{3000} & $\theta_{1}$ &              [0/0] \hfill n/a &      n/a &   n/a &            n/a \\
		& $\theta_{2}$ &              [0/0] \hfill n/a &      n/a &   n/a &            n/a \\
		& $\theta_{3}$ &              [0/0] \hfill n/a &      n/a &   n/a &            n/a \\
		& $\theta_{4}$ &  [3/6] \ \hfill 3334.13(0.02) &  3464.71 &  0.01 &   425.2/1091.0 \\
		& $\theta_{5}$ &  [4/6] \ \hfill 3364.70(0.03) &  3440.82 &  0.01 &   368.2/1120.7 \\
		& $\theta_{6}$ &     [3/3] \ \hfill 3116.18(*) &  3116.18 &  0.00 &   320.3/1188.0 \\
		& $\theta_{7}$ &     [2/2] \ \hfill 3304.96(*) &  3304.96 &  0.01 &    292.0/912.0 \\
		& $\theta_{8}$ &  [1/2] \ \hfill 3189.13(0.03) &  3381.76 &  0.01 &   232.0/1176.0 \\
		& $\theta_{9}$ &     [2/2] \ \hfill 3300.93(*) &  3300.93 &  0.00 &   200.0/1134.0 \\
		& $\theta_{10}$ &     [7/7] \ \hfill 2585.16(*) &  2585.16 &  0.00 &   205.1/1429.6 \\
		\bottomrule
	\end{tabular}
\end{table}


\begin{thebibliography}{35}
\providecommand{\natexlab}[1]{#1}
\providecommand{\url}[1]{\texttt{#1}}
\expandafter\ifx\csname urlstyle\endcsname\relax
  \providecommand{\doi}[1]{doi: #1}\else
  \providecommand{\doi}{doi: \begingroup \urlstyle{rm}\Url}\fi

\bibitem[Abdi and Fukasawa(2016)]{abdi2016mixing-knapsack}
A.~Abdi and R.~Fukasawa.
\newblock On the mixing set with a knapsack constraint.
\newblock \emph{Mathematical Programming}, 157:\penalty0 191--217, 2016.

\bibitem[Atamt{\"u}rk(2006)]{atamturk2006robust}
A.~Atamt{\"u}rk.
\newblock Strong formulations of robust mixed 0-1 programming.
\newblock \emph{Mathematical Programming}, 108:\penalty0 235--250, 2006.

\bibitem[Atamt{\"u}rk et~al.(2000)Atamt{\"u}rk, Nemhauser, and
  Savelsbergh]{atamturk2000mixed}
A.~Atamt{\"u}rk, G.~L. Nemhauser, and M.~W. Savelsbergh.
\newblock The mixed vertex packing problem.
\newblock \emph{Mathematical Programming}, 89\penalty0 (1):\penalty0 35--53,
  2000.

\bibitem[Blanchet and Murthy(2019)]{BlanchetMurthy2019}
J.~Blanchet and K.~Murthy.
\newblock Quantifying distributional model risk via optimal transport.
\newblock \emph{Mathematics of Operations Research}, 44\penalty0 (2):\penalty0
  565--600, 2019.
\newblock \doi{10.1287/moor.2018.0936}.

\bibitem[Calafiore and Campi(2005)]{Cal1}
G.~Calafiore and M.~Campi.
\newblock Uncertain convex programs: randomized solutions and confidence
  levels.
\newblock \emph{Mathematical Programming}, 102:\penalty0 25--46, 2005.

\bibitem[Calafiore and {El Ghaoui}(2006)]{calafiore:jota06}
G.~Calafiore and L.~{El Ghaoui}.
\newblock On distributionally robust chance-constrained linear programs.
\newblock \emph{Journal of Optimization Theory and Applications}, 130:\penalty0
  1--22, 2006.

\bibitem[Campi and Garatti(2011)]{campi}
M.~Campi and S.~Garatti.
\newblock A sampling-and-discarding approach to chance-constrained
  optimization: feasibility and optimality.
\newblock \emph{Journal of Optimization Theory and Applications}, 148:\penalty0
  257--280, 2011.

\bibitem[Chen et~al.(2018)Chen, Kuhn, and Wiesemann]{chen2018data}
Z.~Chen, D.~Kuhn, and W.~Wiesemann.
\newblock Data-driven chance constrained programs over {W}asserstein balls.
\newblock \emph{arXiv:1809.00210}, 2018.

\bibitem[{El Ghaoui} et~al.(2003){El Ghaoui}, Oks, and Oustry]{EOO03}
L.~{El Ghaoui}, M.~Oks, and F.~Oustry.
\newblock Worst-case value-at-risk and robust portfolio optimization: A conic
  programming approach.
\newblock \emph{Operations Research}, 51\penalty0 (4):\penalty0 543--556, 2003.

\bibitem[Gao and Kleywegt(2016)]{gao2016distributionally}
R.~Gao and A.~J. Kleywegt.
\newblock Distributionally robust stochastic optimization with {Wasserstein}
  distance.
\newblock \emph{arXiv:1604.02199}, 2016.

\bibitem[G{\"u}nl{\"u}k and Pochet(2001)]{gunluk2001mixing}
O.~G{\"u}nl{\"u}k and Y.~Pochet.
\newblock Mixing mixed-integer inequalities.
\newblock \emph{Mathematical Programming}, 90\penalty0 (3):\penalty0 429--457,
  2001.

\bibitem[Hanasusanto et~al.(2015)Hanasusanto, Roitch, Kuhn, and
  Wiesemann]{Hanasusanto2015}
G.~A. Hanasusanto, V.~Roitch, D.~Kuhn, and W.~Wiesemann.
\newblock A distributionally robust perspective on uncertainty quantification
  and chance constrained programming.
\newblock \emph{Mathematical Programming}, 151\penalty0 (1):\penalty0 35--62,
  2015.

\bibitem[{Hota} et~al.(2019){Hota}, {Cherukuri}, and {Lygeros}]{HotaEtAl2019}
A.~R. {Hota}, A.~{Cherukuri}, and J.~{Lygeros}.
\newblock Data-driven chance constrained optimization under {Wasserstein}
  ambiguity sets.
\newblock In \emph{2019 American Control Conference (ACC)}, pages 1501--1506,
  July 2019.
\newblock \doi{10.23919/ACC.2019.8814677}.

\bibitem[Ji and Lejeune(2019)]{JiLejeune2019}
R.~Ji and M.~Lejeune.
\newblock Data-driven distributionally robust chance-constrained optimization
  with {Wasserstein} metric.
\newblock Technical report, April 2019.
\newblock \url{http://www.optimization-online.org/DB_HTML/2018/07/6697.html}.

\bibitem[Jiang and Guan(2016)]{JG16}
R.~Jiang and Y.~Guan.
\newblock Data-driven chance constrained stochastic program.
\newblock \emph{Mathematical Programming}, 158:\penalty0 291--327, 2016.

\bibitem[K{\i}l{\i}n{\c{c}}-Karzan et~al.(2019)K{\i}l{\i}n{\c{c}}-Karzan,
  K{\"u}{\c{c}}{\"u}kyavuz, and Lee]{Kilinc-Karzan2019joint-sumod}
F.~K{\i}l{\i}n{\c{c}}-Karzan, S.~K{\"u}{\c{c}}{\"u}kyavuz, and D.~Lee.
\newblock Joint chance-constrained programs and the intersection of mixing sets
  through a submodularity lens.
\newblock \emph{arXiv:1910.01353}, 2019.

\bibitem[K{\"u}{\c{c}}{\"u}kyavuz(2012)]{kucukyavuz2012mixing}
S.~K{\"u}{\c{c}}{\"u}kyavuz.
\newblock On mixing sets arising in chance-constrained programming.
\newblock \emph{Mathematical Programming}, 132\penalty0 (1-2):\penalty0 31--56,
  2012.

\bibitem[K{\"u}{\c{c}}{\"u}kyavuz and Noyan(2016)]{kuccukyavuz2016cut}
S.~K{\"u}{\c{c}}{\"u}kyavuz and N.~Noyan.
\newblock Cut generation for optimization problems with multivariate risk
  constraints.
\newblock \emph{Mathematical Programming}, 159\penalty0 (1-2):\penalty0
  165--199, 2016.

\bibitem[Li et~al.(2019)Li, Jiang, and Mathieu]{Li2019}
B.~Li, R.~Jiang, and J.~L. Mathieu.
\newblock Ambiguous risk constraints with moment and unimodality information.
\newblock \emph{Mathematical Programming}, 173\penalty0 (1):\penalty0 151--192,
  Jan 2019.

\bibitem[Liu et~al.(2016)Liu, K{\"u}{\c{c}}{\"u}kyavuz, and Luedtke]{LKL16}
X.~Liu, S.~K{\"u}{\c{c}}{\"u}kyavuz, and J.~Luedtke.
\newblock Decomposition algorithms for two-stage chance-constrained programs.
\newblock \emph{Mathematical Programming}, 157\penalty0 (1):\penalty0 219--243,
  2016.

\bibitem[Liu et~al.(2017)Liu, K{\"u}{\c{c}}{\"u}kyavuz, and Noyan]{Liu2017}
X.~Liu, S.~K{\"u}{\c{c}}{\"u}kyavuz, and N.~Noyan.
\newblock Robust multicriteria risk-averse stochastic programming models.
\newblock \emph{Annals of Operations Research}, 259\penalty0 (1):\penalty0
  259--294, 2017.

\bibitem[Liu et~al.(2019)Liu, K{\i}l{\i}n{\c{c}}-Karzan, and
  K{\"u}{\c{c}}{\"u}kyavuz]{liu2018intersection}
X.~Liu, F.~K{\i}l{\i}n{\c{c}}-Karzan, and S.~K{\"u}{\c{c}}{\"u}kyavuz.
\newblock On intersection of two mixing sets with applications to joint
  chance-constrained programs.
\newblock \emph{Mathematical Programming}, 175:\penalty0 29--68, 2019.

\bibitem[Luedtke(2014)]{luedtke2014branch-and-cut}
J.~Luedtke.
\newblock A branch-and-cut decomposition algorithm for solving
  chance-constrained mathematical programs with finite support.
\newblock \emph{Mathematical Programming}, 146:\penalty0 219--244, 2014.

\bibitem[Luedtke and Ahmed(2008)]{sample}
J.~Luedtke and S.~Ahmed.
\newblock A sample approximation approach for optimization with probabilistic
  constraints.
\newblock \emph{SIAM Journal on Optimization}, 19:\penalty0 674--699, 2008.

\bibitem[Luedtke et~al.(2010)Luedtke, Ahmed, and Nemhauser]{luedtke2010integer}
J.~Luedtke, S.~Ahmed, and G.~L. Nemhauser.
\newblock An integer programming approach for linear programs with
  probabilistic constraints.
\newblock \emph{Mathematical Programming}, 122\penalty0 (2):\penalty0 247--272,
  2010.

\bibitem[Merakl{\i} and K{\"u}{\c{c}}{\"u}kyavuz(2019)]{Merakli2019}
M.~Merakl{\i} and S.~K{\"u}{\c{c}}{\"u}kyavuz.
\newblock Risk aversion to parameter uncertainty in {Markov} decision processes
  with an application to slow-onset disaster relief.
\newblock \emph{IISE Transactions}, pages 1--21, 2019.
\newblock \doi{10.1080/24725854.2019.1674464}.
\newblock Article in advance.

\bibitem[Mohajerin~Esfahani and Kuhn(2018)]{MohajerinEsfahaniKuhn2018}
P.~Mohajerin~Esfahani and D.~Kuhn.
\newblock Data-driven distributionally robust optimization using the
  wasserstein metric: performance guarantees and tractable reformulations.
\newblock \emph{Mathematical Programming}, 171\penalty0 (1):\penalty0 115--166,
  Sep 2018.
\newblock \doi{10.1007/s10107-017-1172-1}.

\bibitem[Noyan(2018)]{Noyan2018tutorial}
N.~Noyan.
\newblock Risk-averse stochastic modeling and optimization.
\newblock In \emph{{TutORials} in Operations Research: Recent Advances in
  Optimization and Modeling of Contemporary Problems}, chapter~10, pages
  221--254. INFORMS, 2018.
\newblock \doi{10.1287/educ.2018.0183}.

\bibitem[Noyan et~al.(2019)Noyan, Merakl{\i}, and
  K{\"u}{\c{c}}{\"u}kyavuz]{Noyan2019}
N.~Noyan, M.~Merakl{\i}, and S.~K{\"u}{\c{c}}{\"u}kyavuz.
\newblock Two-stage stochastic programming under multivariate risk constraints
  with an application to humanitarian relief network design.
\newblock \emph{Mathematical Programming}, pages 1--39, 2019.
\newblock \doi{10.1007/s10107-019-01373-4}.
\newblock Article in advance.

\bibitem[Rahimian and Mehrotra(2019)]{Rahimian2019DistributionallyRO}
H.~Rahimian and S.~Mehrotra.
\newblock Distributionally robust optimization: A review.
\newblock \emph{arXiv:1908.05659}, 2019.

\bibitem[Ruszczy\'{n}ski(2002)]{ruspp:02}
A.~Ruszczy\'{n}ski.
\newblock Probabilistic programming with discrete distributions and precedence
  constrained knapsack polyhedra.
\newblock \emph{Mathematical Programming}, 93:\penalty0 195--215, 2002.

\bibitem[Xie(2019)]{xie2018distributionally}
W.~Xie.
\newblock On distributionally robust chance constrained programs with
  {Wasserstein} distance.
\newblock \emph{Mathematical Programming}, 2019.
\newblock \doi{10.1007/s10107-019-01445-5}.
\newblock Article in advance.

\bibitem[Xie and Ahmed(2018{\natexlab{a}})]{Xie18}
W.~Xie and S.~Ahmed.
\newblock On deterministic reformulations of distributionally robust joint
  chance constrained optimization problems.
\newblock \emph{SIAM Journal on Optimization}, 28\penalty0 (2):\penalty0
  1151--1182, 2018{\natexlab{a}}.

\bibitem[Xie and Ahmed(2018{\natexlab{b}})]{xie2018quantile}
W.~Xie and S.~Ahmed.
\newblock On quantile cuts and their closure for chance constrained
  optimization problems.
\newblock \emph{Mathematical Programming}, 172:\penalty0 621--646,
  2018{\natexlab{b}}.

\bibitem[Zhao et~al.(2017)Zhao, Huang, and Zeng]{zhao2017joint-knapsack}
M.~Zhao, K.~Huang, and B.~Zeng.
\newblock A polyhedral study on chance constrained program with random
  right-hand side.
\newblock \emph{Mathematical Programming}, 166:\penalty0 19--64, 2017.

\end{thebibliography}
\end{document}